\documentclass[12pt]{amsart}
\usepackage{graphicx}
\usepackage{here}
\usepackage{tikz}

\usepackage{amsmath, amssymb}
\usepackage{amscd}
\usepackage{pifont}
\usepackage{booktabs}

\usepackage[all]{xy}

\numberwithin{equation}{section}
\newtheorem{theorem}{Theorem}[section]  
\newtheorem{theorem?}{``Theorem''}[section]  

\newtheorem{proposition}[theorem]{Proposition}
\newtheorem{lemma}[theorem]{Lemma}

\newtheorem{problem}{Problem}[section]

\theoremstyle{definition}
\newtheorem{definition}[theorem]{Definition}

\theoremstyle{remark}
\newtheorem{remark}[theorem]{Remark}  

\newcommand{\R}{{\mathbb R}}
\newcommand{\C}{{\mathbb C}}

\newcommand{\N}{{\mathbb N}}
\newcommand{\Z}{{\mathbb Z}}
\renewcommand{\a}{\alpha}
\renewcommand{\b}{\beta}
\renewcommand{\d}{\partial}

\begin{document}
\title[Resolution of singularities and local zeta functions]
{
Resolution of singularities for $C^{\infty}$ functions \\
and meromorphy of local zeta functions
} 
\author{Joe Kamimoto}
\address{Faculty of Mathematics, Kyushu University, 
Motooka 744, Nishi-ku, Fukuoka, 819-0395, Japan} 
\email{
joe@math.kyushu-u.ac.jp}
\dedicatory{To the memory of Professor Masatake Kuranishi.}
\maketitle


\begin{abstract}
In this paper, we attempt to resolve the singularities of 
the zero variety of a $C^{\infty}$ function of two variables
as much as possible by using ordinary blowings up. 
As a result, we formulate an algorithm 
to locally express the zero variety 
in the ``almost'' normal crossings form, 
which is close to the normal crossings 
form but may include flat functions.
As an application, 
we investigate analytic continuation of local zeta functions associated
with $C^{\infty}$ functions of two variables. 
As is well known, 
the desingularization theorem of Hironaka 
implies that the local zeta functions associated with real analytic functions
admit the meromorphic continuation to the whole complex plane. 
On the other hand, 
it is recently observed that 
the local zeta function associated with
a specific (non-real analytic) $C^{\infty}$ function 
has a singularity different from the pole. 
From this observation, 
the following questions are naturally raised 
in the $C^{\infty}$ case:
how wide the meromorphically extendible region can be 
and 
what kinds of information essentially determine this region?
This paper shows
that this region can be described in terms of 
some kind of multiplicity of the zero variety of 
each $C^{\infty}$ function. 
By using our blowings up algorithm, 
it suffices to investigate local zeta functions 
in the almost normal crossings case.
This case can be effectively analyzed 
by using real analysis methods;  
in particular, 
a van der Corput-type lemma plays a crucial role
in the determination of the above region. 
\end{abstract}




\section{Introduction}

In this paper, we study an integral of the form
\begin{equation}\label{eqn:1.1}
Z(f,\varphi)(s):=
\int_{\R^2}|f(x,y)|^s \varphi(x,y)dxdy 
\,\,\,\,\,\quad s\in \C,
\end{equation}
where 
$f, \varphi$ are real-valued $C^{\infty}$ functions defined 
on a small open neighborhood $U$ of the origin in $\R^2$ and  
the support of $\varphi$ is contained in $U$.
Since the integral in (\ref{eqn:1.1})
locally converges on the region
${\rm Re}(s)>0$,
$Z(f,\varphi)$ can be regarded as a holomorphic function there, 
which is called a {\it local zeta function}.
We are interested in an issue: 
how local zeta functions can be analytically continued 
to a wider region. 
Since the analytic continuation issue 
has multiple connections with many mathematics, 
such as partial differential equations, complex anaysis,
harmonic analysis, number theory, representation theory, 
and singularity theory among others, 
the theory of local zeta functions has been considerably evolved
(c.f. \cite{Mal74}, \cite{AGV88}, \cite{Den91}, \cite{Igu00}, etc.). 


Observing the form of the integral in (\ref{eqn:1.1}), 
we can see that 
the convergence of the integral induces the holomorphic extension
of local zeta functions
and, moreover, that
their analytic continuation 
is deeply related to the geometry of the zero variety of $f$.
When $f$ is real analytic,
it was shown 
by Bernstein and Gel'fand \cite{BeG69} and M. Atiyah \cite{Ati70} 
that local zeta functions 
admit a meromorphic continuation to the whole complex plane. 
The most important idea of their works is to locally express 
the zero variety of $f$ in the normal crossings form 
by using the desingularization theorem of Hironaka \cite{Hir64}.
Their results were the first remarkable 
applications of Hironaka's theorem to an important analytic issue.

More generally, 
understanding the geometric structure of the zero variety of  
a multivariate function is crucial in many 
important issues in analysis; 
harmonic analysis, parital differential equations, 
complex analysis, probability, etc. 
Since the works \cite{BeG69}, \cite{Ati70}, 
resolution of singularities has been recognized to 
be a powerful tool for these issues.
It should be specially mentioned that
reasonable resolutions of singularities 
recently give many strong results concerning harmonic analysis
(\cite{PhS97}, \cite{PSS99}, \cite{PhS00}, \cite{Gre04}, 
\cite{Gre08}, \cite{Gre10jam}, \cite{IKM10}, 
\cite{IkM11jfaa}, \cite{CGP13}, \cite{IkM16}, etc.). 
For the application of desingularization theorems, 
some kind of analyticity of the corresponding function 
is usually required, but 
it is sometimes desirable to deal with 
a given issue in a more general setting. 
Therefore, 
it is meaningful to try to improve desingularization theorems 
for a wider class of functions. 
For example, in the study of oscillatory integrals and 
local zeta functions in \cite{KaN16jmst},  
since desingularization theorems have been established 
in a certain class of $C^{\infty}$ functions including 
the Denjoy-Carleman quasianalytic class
(see also \cite{BiM97}, \cite{BiM04}) 
and, as a result,   
many strong results obtained in the real analytic case 
can be generalized by using these theorems. 

Unfortunately, 
since there exist  $C^{\infty}$ functions whose singularities cannot 
be completely resolved by using algebraic transforms only,  
the general $C^{\infty}$ case is hard to deal with 
from a geometrical point of view. 
Furthermore, a distinctive phenomenon concerning  
these troublesome $C^{\infty}$ functions 
has been recently observed in an analytic issue for local zeta functions. 
To be more specific, 
it was shown in \cite{KaN19} 
(see also \cite{Gre06}) that 
the local zeta function associated with 
a specific (non-real analytic) $C^{\infty}$ function 
cannot be meromorphically
extended to the whole complex plane (see Section~2.3). 
In other words, 
local zeta functions possibly have singularities different from poles.
Then a new issue is naturally raised in the $C^{\infty}$ case:
how wide the meromorphically extendible region of local zeta functions 
can be and 
what kinds of information of a $C^{\infty}$ function $f$ 
essentially determine this region? 
(This issue will be more exactly formulated 
in Section~2.3.)
For the investigation of this issue, 
it is necessary to understand geometrical properties of the zero varieties 
of general $C^{\infty}$ functions.  
This geometrical issue itself seems interesting and important 
from various motivations. 
The first half of this paper is devoted to the investigation 
of this issue.  

Let $f$ be a $C^{\infty}$ function defined 
near the origin in $\R^2$.
When the set defined by $f(x,y)=0$ is 
restricted to the real space $\R^2$,  
this restricted set sometimes has very few information 
and is not always useful for precise analysis. 
In the case where $f$ is real analytic, the defining region of 
$f$ can be naturally extended to the complex region in $\C^2$. 
The zero variety in $\C^2$ of the extended $f$ is 
so-called a {\it holomorphic plane curve},  
which has been very widely studied.  
Actually, 
many fruitful results about these curves 
improve the investigation of local zeta functions 
associated with real analytic functions. 
For example, the theory of toric varieties 
based on the geometry of Newton polyhedra
gives quantitative results 
about poles of local zeta functions
(\cite{Var76}, \cite{DeS89}, \cite{DeS92}, 
\cite{DNS05}, \cite{OkT13}, \cite{CKN13}, 
\cite{KaN16jmst}, \cite{KaN16tams}, etc.).
On the other hand, 
when a $C^{\infty}$ function $f$ is extended to the complex space,  
the conjugate variables must be considered in general, which 
makes it difficult to understand geometric properties of 
the zero variety of $f$ in $\C^2$. 
Therefore, we give up handling this variety itself and  
istead look for an essentially important subset in it, 
which is easier to deal with.  
With the aid of the factorization formula for 
$C^{\infty}$ functions of V. S. Rychkov \cite{Ryc01}, 
an important {\it curve} in the zero locus of $f$ in $\C^2$ 
is defined, 
which will be called the {\it decisive curve}
defined by $f$, 
and this curve has sufficient information for our analysis.   
The decisive curve defined by $f$
consist of branches in $\C^2$ 
parametrized by using the Puiseux series
of one real variable. 
Although the singularity of this curve might not be
completely resolved by using algebraic transforms only, 
this curve can be locally expressed as in {\it almost} 
normal crossings form 
via finite compositions of ordinary blowings up. 
To be more exact, 
there exist a two-dimensional $C^{\infty}$ real manifold $Y$ 
and a proper map $\pi:Y\to \R^2$ such that $f\circ\pi$ 
can be locally expressed at any point on the zero locus of 
the map $\pi$ as 
\begin{equation}\label{eqn:1.2}
(f\circ\pi)(x,y)=u(x,y) x^a 
\left(
y^m+\varepsilon_1(x) y^{m-1}+\cdots+\varepsilon_m(x)
\right),
\end{equation}
where $a,m$ are nonnegative integers and 
$u, \varepsilon_k$ are real-valued $C^{\infty}$ functions 
satisfying that $u(0,0)\neq 0$ 
and $\varepsilon_k$ are flat at the origin. 
Note that in the real analytic case, 
since $\varepsilon_k$ must be zero functions, 
$f\circ\pi$ can be locally expressed 
in ordinary normal crossings form, 
which implies that each local zeta function
can be meromorphically extended to the whole 
complex plane by using an elementary method 
(\cite{AGV88}, see also Section~11). 

After a desingularization theorem was shown 
by Hironaka \cite{Hir64}, 
more elementary constructive proofs for the theorem
have been given
(\cite{Sus90}, \cite{BiM97},  etc.), which reveal  
the situation of resolution of singularities more clearly 
and make the theorems more applicable to the analysis.  
Our method is in this direction and 
the proof in this paper is self-contained
if the result of Rychkov in \cite{Ryc01} is agreed. 
%
It is expected that our {\it almost} desingularization theorem 
will be established in the general dimensional case and 
will be usefully applied to various analytic issues.

Throughout the above geometric process, 
it is sufficient to deal with 
local zeta functions associated with functions of the form
(\ref{eqn:1.2}), 
which is considered as a model in the $C^{\infty}$ case. 
The latter half of this paper is devoted to the investigation 
of the analytic continuation of local zeta functions 
in this model case.  
In this case, these analytic continuation 
can be effectively investigated by using real analysis methods;
the most important tool is a van der Corput-type lemma. 
The original van der Corput's lemma gives an estimate 
for one-dimensional oscillatory integrals, 
which is explained in \cite{Ste93}. 
This lemma has been rewritten in various forms 
according to the purposes. 
Our analysis needs one of the versions used in \cite{Chr85} 
(see also \cite{Gre06}). 
Note that these van der Corput type lemmas play key roles 
in recent studies in harmonic analysis,
which investigate estimations of oscillatory integral operators, 
oscillatory integrals, 
Fourier restrictions, maximal operators, 
critical integrability indices,
the sizes of sublevel sets and so on
((\cite{PhS97}, \cite{PSS99}, \cite{PhS00},
 \cite{Gre10jam}, \cite{IKM10}, 
\cite{IkM11jfaa}, \cite{CGP13}, \cite{IkM16}, etc.). 
As a result, we show that
the meromorphically extendible region of local zeta functions
associated with (\ref{eqn:1.2}) 
contains the region ${\rm Re}(s) > -1/m$. 
The above mentioned analysis has been essentially performed 
in the recent paper \cite{KaN20}.

After the above explained investigation, 
we give an answer to 
the meromorphic extension issue
for local zeta functions in the $C^{\infty}$ case. 
For this purpose,  we introduce 
a quantity $\mu_0(f)$ for a given
$C^{\infty}$ function $f$.
In general, the double formal power series 
has the factorization formula by using the Puiseux series. 
Through the above explained resolution process, 
the multiplicities of real roots in this factorization formula 
essentially appear in the index $m$ in 
the expression (\ref{eqn:1.2}).
The maximum of the multiplicities of real roots 
in the factorization formula is denoted by $\mu_0(f)$.
Then we can see that the meromorphically extendible region always 
contains the region ${\rm Re}(s)>-1/\mu_0(f)$.
This result is optimal in the uniform sense. 
Note that the quantity $\mu_0(f)$ is an invariant of $f$, 
i.e., it is independent of the choice
of coordinates (see Section~8.3).

This paper is organized as follows. 
In Section~2, after exactly describing our analytic issues
of local zeta functions and introducing the quantity $\mu_0(f)$, 
we state our main theorem. 
Sections~3--8 are the geometrical part of this paper. 
In Section~3, we state the most important theorem
from a geometrical point of view, 
which gives an almost resolution of singularities 
for $C^{\infty}$ functions. 
In Section~4, we recall process of 
blowings up, which is well-known 
in the study of algebraic geometry, etc. 
In Section~5, we explain an important factorization formula
for $C^{\infty}$ functions given by Rychkov \cite{Ryc01}.  
By using many tools in Sections 4-5, 
we actually attempt to resolve the singularities of the zero variety 
of $C^{\infty}$ functions as much as possible 
and, as a result, we obtain a desired desingularization theorem. 
In Section~7, 
we give a proof of the theorem stated in Section~3.
Since the quantity $\mu_0(f)$ used in the theorem
plays important roles in the analytic continuation 
of local zeta functions, we precisely investigate
its properties in Section 8.
Sections~9--13 are the analytic part of this paper. 
In Section 9, 
after using almost resolution of singularities and decomposing 
the integral in (\ref{eqn:1.1}), 
we can see that 
it suffices to consider the model case as in (\ref{eqn:1.2}).
We state the most important result 
from an analytical point of view,   
which describes the meromorphically extendible region 
in the model case. 
In Section 10,  we give a proof of the main theorem stated in Section~2
by using results in Section~9.
In Section~11, we prepare many useful analytic 
lemmas for the subsequent analysis. 
In Section~12, geometric and analytic properties
of the above model function are investigated. 
In Section~13, we give a proof for the theorem stated in Section~9
by using many tools prepared in Sections 11-12. 
At present, 
there have been very few results about the meromorphic extension issue 
of local zeta functions  
and there are many open issues which should be investigated in the future. 
We discuss these matters in Section~14.  

\vspace{.5 em}
{\it Notation and symbols.}\quad

\begin{itemize}
\item 
We denote by $\Z_+, \R_+$ 
the subsets consisting of 
all nonnegative numbers 
in $\Z,\R$, respectively.
For $s\in\C$, ${\rm Re}(s)$ expresses the real part of $s$.
\item
For $n\in\N$, 
we denote by ${\mathbb P}^n(\C)$ 
(or ${\mathbb P}^n(\R)$) 
the $n$-dimensional complex projective space 
(or real projective space). 
\item 
For  $R=\R$ or $\C$, 
$R[t]$, $R[[t]]$, $R\{t\}$  are the rings of
polynomials, 
formal power series, 
convergent power series in $t$ with 
coefficients from $R$, respectively.
Moreover, $R[[x,y]]$, $R\{x,y\}$  
are the rings of double formal power
series and double convergent power series, respectively. 
\item 
For an open set $U$ in $\R^2$,
 $C^{\omega}(U)$ denotes the set of 
real analytic functions on $U$.
\item
By (\ref{eqn:1.1}), $Z(f,\varphi)(s)$ is defined as an {\it integral}.  
When $Z(f,\varphi)$ can be regarded as a {\it function} on 
some region, 
this function is also denoted by the same symbol.
\end{itemize}


\section{Description of the problems and the main result}

Let $U$ be a small open neighborhood of the origin in $\R^2$
and 
let $f,\varphi\in C^{\infty}(U)$ satisfy the conditions 
in the Introduction. 
Moreover, we usually assume that 
$f\in C^{\infty}(U)$ is non-flat and satisfies 
\begin{equation}\label{eqn:2.1}
f(0,0)=0 \quad \mbox{and} \quad \nabla f (0,0)=(0,0).
\end{equation}
Unless (\ref{eqn:2.1}) is satisfied,  
every problem addressed in this paper is easy.
As for $\varphi\in C_0^{\infty}(U)$, 
we sometimes give the following conditions
\begin{equation}\label{eqn:2.2}
\varphi(0,0)>0 \mbox{\,\,\, and \,\,\,} \varphi\geq 0
\mbox{ on $U$}.
\end{equation}
In order to investigate the analytic continuation of 
local zeta functions, 
we only use the half-plane of the form
${\rm Re}(s)>-\rho$
with $\rho>0$.
This is the reason why we observe the situation of analytic continuation
through the integrability 
of integrals of the form (\ref{eqn:1.1}). 
Of course, it is desirable to deal with various kinds of regions 
in the study of analytic continuation and 
this advanced issue should be investigated in the future.

\subsection{Newton data}

Let $\overline{f}(x,y)\in \R[[x,y]]$ 
be the Taylor series of 
$f(x,y)$ at the origin, i.e., 
\begin{equation}\label{eqn:2.3}
\overline{f}(x,y)= 
\sum_{(j,k)\in{\Z}_+^2} 
c_{jk}x^{j}y^{k} 
\quad \mbox{ with $c_{jk}=
\dfrac{1}{j! k!}
\dfrac{\partial^{j+k} f}
{\partial x^{j}\partial y^{k}}(0,0)$}.
\end{equation}
The {\it Newton polygon} of $f$
is the integral polygon
$$
\Gamma_+(f)=
\mbox{the convex hull of the set 
$\bigcup \{(j,k)+\R_+^2:
c_{jk}\neq 0\}$ in $\R_+^2$}
$$
(i.e., the intersection 
of all convex sets 
which contain 
$\bigcup \{(j,k)+\R_+^2:c_{jk}\neq 0\}$). 
The flatness of $f$ at the origin is equivalent to 
the condition
$\Gamma_+(f)=\emptyset$.


The {\it Newton distance} $d(f)$ of $f$ is defined by 
\begin{equation*}
d(f)=\inf\{\alpha>0:(\alpha,\alpha)\in\Gamma_+(f)\}.
\end{equation*}
We set $d(f)=\infty$ when $f$ is flat at the origin.
Since the Newton distance depends on the coordinates system 
$(x,y)$ on which $f$ is defined, 
it is sometimes denoted by $d_{(x,y)}(f)$.
The {\it height} of $f$
is defined by 
\begin{equation}\label{eqn:2.4}
\delta_0(f)=\sup_{(x,y)}\{d_{(x,y)}(f)\},
\end{equation}
where the supremum is taken over all local smooth coordinate 
systems $(x,y)$ at the origin. 
A given coordinate system $(x,y)$ is said to be {\it adapted} to
$f$, if the equality $\delta_0(f)=d_{(x,y)}(f)$ holds. 
Note that the height $\delta_0(f)$ can be determined
by the Taylor series $\overline{f}\in\R[[x,y]]$ only.
From their definitions, 
$d(f)$ and $\delta_0(f)$ roughly indicate some kind of 
flatness of $f$ at the origin 
(when they are larger, the flatness of $f$ becomes stronger). 

\begin{remark}
(1)\quad 
We can determine $\delta_0(f)$ for $f$ not satisfying 
the conditions (\ref{eqn:2.1}) from its definition. 
When $f(0,0)\neq 0$, we have $\delta_0(f)=0$. 
When $f(0,0)=0$ and $\nabla f(0,0)\neq (0,0)$, 
we have $\delta_0(f)=1$ 
by using the implicit function theorem.

(2)\quad 
The existence of adapted coordinates 
(in the two-dimensional case)
is shown in \cite{Var76}, \cite{PSS99}, 
\cite{IkM11tams}, etc.
Furthermore, 
useful necessary and sufficient conditions 
for their adaptedness have been obtained 
in \cite{Var76}, \cite{AGV88}, \cite{IkM11tams} 
(they will be explained in Section~8.4). 
We remark that 
the existence of adapted coordinates is not obvious. 
The definition of 
the adapted coordinate can be directly 
generalized in higher dimensional case.
In the three-dimensional case, 
it is known in \cite{Var76} that 
there exists a real analytic 
function admitting no adapted coordinate.  
\end{remark}

\subsection{Holomorphic extension problem}
First, let us consider the following quantities:
\begin{equation}\label{eqn:2.5}
{\mathfrak h}_0(f,\varphi):=\sup\left\{\rho>0: 
\begin{array}{l} 
\mbox{The domain to which $Z(f,\varphi)$ can} \\ 
\mbox{be holomorphically continued} \\
\mbox{contains the half-plane ${\rm Re}(s)>-\rho$}
\end{array}
\right\}, 
\end{equation}
\begin{equation}\label{eqn:2.6}
{\mathfrak h}_0(f):=\inf\left\{{\mathfrak h}_0(f,\varphi):
\varphi\in C_0^{\infty}(U)\right\}.
\end{equation}
It is obvious that ${\mathfrak h}_0(f)$ is invariant 
under the change of coordinates. 
We remark that 
if $\varphi$ satisfies (\ref{eqn:2.2}), 
then ${\mathfrak h}_0(f,\varphi) = 
{\mathfrak h}_0(f)$ holds; 
but
otherwise, 
this equality does not always hold.
Indeed,
there exists $\varphi\in C_0^{\infty}(U)$ with 
$\varphi(0,0)=0$
such that ${\mathfrak h}_0(f,\varphi)>
{\mathfrak h}_0(f)$ 
(see e.g. \cite{CKN13}, \cite{KaN16tams}).

From the form of the integral in (\ref{eqn:1.1}),
the relationship between the holomorphy and 
the convergence of the integral implies that  
the quantity ${\mathfrak h}_0(f)$ is deeply related to
the following famous index:
\begin{equation}\label{eqn:2.7}
{\mathfrak c}_0(f):=\sup
\left\{\mu>0:
\begin{array}{l} 
\mbox{there exists an open neighborhood $V$ of} \\
\mbox{the origin in $U$ such that $|f|^{-\mu}\in L^1(V)$}
\end{array}
\right\},
\end{equation}
which is called the
{\it log canonical threshold} or 
the {\it critical integrability index}. 
The index ${\mathfrak c}_0(f)$ 
has been deeply investigated from various points of view. 
The equality ${\mathfrak h}_0(f)={\mathfrak c}_0(f)$ always holds. 
In fact, 
the inequality ${\mathfrak h}_0(f)\geq {\mathfrak c}_0(f)$ is obvious; 
while the opposite inequality can be easily seen
by Theorem~5.1 in \cite{KaN19}.
In the real analytic case, 
since all the singularities of the extended $Z(f,\varphi)$ 
are poles on the real axis, 
the leading pole exists at $s=-{\mathfrak h}_0(f,\varphi)$.
In the seminal work of Varchenko \cite{Var76},
when $f$ is real analytic and satisfies some nondegeneracy conditions
(see Section~8.4, below),  
${\mathfrak h}_0(f)$ can be expressed as
${\mathfrak h}_0(f)=1/d(f)$, where $d(f)$ is the Newton distance
of $f$.
An interesting work \cite{CGP13} treating the equality 
${\mathfrak c}_0(f)=1/d(f)$ is from another approach.
We remark that these results deal with the general dimensional case. 
In the same paper \cite{Var76}, 
Varchenko more deeply investigates the two-dimensional case. 
Indeed, without any assumption, he shows that 
the equality 
\begin{equation}\label{eqn:2.8}
{\mathfrak h}_0(f)=1/\delta_0(f)
\end{equation} 
holds for real analytic $f$.
More generally, in the $C^{\infty}$ case, 
M. Greenblatt \cite{Gre06} obtains a sharp result 
which generalizes the above two-dimensional result of Varchenko. 
\begin{theorem}[\cite{Gre06}]
${\mathfrak c}_0(f) (={\mathfrak h}_0(f))=1/\delta_0(f)$ holds 
for every non-flat $f\in C^{\infty}(U)$.
\end{theorem}

From the above result, 
our holomorphic extension problem is completely understood 
even in the $C^{\infty}$ case. 
It is important that 
${\mathfrak h}_0(f)$ is determined by information 
of the formal Taylor series of $f$ only.

On the other hand, the situation of
the meromorphic extension 
is quite different from the holomorphic one.

\subsection{Meromorphic extension problem}

Corresponding to (\ref{eqn:2.5}), (\ref{eqn:2.6})
in the holomorphic continuation case,
we analogously define the following quantities:
\begin{equation}\label{eqn:2.9}
{\mathfrak m}_0(f,\varphi):=\sup\left\{\rho>0: 
\begin{array}{l} 
\mbox{The domain to which $Z(f,\varphi)$ can} \\ 
\mbox{be meromorphically continued} \\
\mbox{contains the half-plane ${\rm Re}(s)>-\rho$}
\end{array}
\right\}, 
\end{equation}
\begin{equation}\label{eqn:2.10}
{\mathfrak m}_0(f):=\inf\left\{{\mathfrak m}_0(f,\varphi):
\varphi\in C_0^{\infty}(U)\right\}.
\end{equation}
It is easy to see that ${\mathfrak m}_0(f)$ is invariant 
under the change of coordinates and 
that 
${\mathfrak h}_0(f,\varphi)\leq {\mathfrak m}_0(f,\varphi)$ and
${\mathfrak h}_0(f)\leq {\mathfrak m}_0(f)
\leq {\mathfrak m}_0(f,\varphi)$ 
 always hold.
As mentioned in the Introduction, 
if $f$ is real analytic, then 
${\mathfrak m}_0(f)=\infty$ always holds; while
there exist specific (non-real analytic) $C^{\infty}$ functions $f$ 
such that ${\mathfrak m}_0(f)<\infty$. 
Indeed, it is shown in \cite{KaN19} (see also \cite{Gre06}) 
that when
\begin{equation}\label{eqn:2.11}
f(x,y)=x^a y^b+x^ay^{b-q}e^{-1/|x|^p},
\end{equation}
and
$\varphi$ satisfies the condition \eqref{eqn:2.2}, 
$Z(f,\varphi)$ has a non-polar singularity at 
$s=-1/b$,
which implies ${\mathfrak m}_0(f)=1/b$.
Here, $p$ is a positive real number and 
$a, b, q \in\Z_+$ satisfy that
$a<b$,  $b\geq 2$,  $1\leq q\leq b$ 
and $q$ is even.
Note that $d(f)=\delta_0(f)=b$ in this case. 
At present, properties of the singularity 
at $s=-1/b$ are not well understood
(see Section~14.2).
In order to understand how wide 
the meromorphically extendible region 
of a given local zeta function is, 
we consider the following problem.

\begin{problem}
For a given $f\in C^{\infty}(U)$, 
describe (or estimate) 
the value of ${\mathfrak m}_0(f)$ in terms of appropriate
information of $f$.
\end{problem}

In \cite{KaN19}, 
the above problem is investigated  
in the case where $f$ has the following form
which is a natural generalization of (\ref{eqn:2.11}).
\begin{equation}\label{eqn:2.12}
f(x,y)=u(x,y)x^a y^b
+(\mbox{a flat function}),
\end{equation}
where $a, b$ are nonnegative integers with $a\leq b$  and 
$u(x,y)\in C^{\infty}(U)$ satisfies
$u(0,0)\neq 0$. 
It is shown in \cite{KaN19} that ${\mathfrak m}_0(f)\geq 1/b$.
Note that $\delta_0(f)=b$ in this case. 

\begin{remark}
Since $x^a y^b$ with $a,b\in\Z_+$ is a real analytic function,
${\mathfrak m}_0(x^a y^b)=\infty$ holds. 
On the other hand, 
${\mathfrak m}_0(f)=1/b$ holds 
if $f$ is as in (\ref{eqn:2.11}). 
From this observation,  
we see that ${\mathfrak m}_0(f)$ is not always 
determined by the formal Taylor series of $f$.  
\end{remark}

\subsection{The quantity $\mu_0(f)$}


Let us introduce an important quantity 
$\mu_0(f)$, which will be used in the statement 
of the main theorem.  

Let $\overline{f}(x,y)\in\R [[x,y]]$ 
be the formal Taylor series of a non-flat $C^{\infty}$ 
function $f(x,y)$ at the origin. 
It is known (c.f. \cite{Wal04}, Corollary~2.4.2, p.32) that 
$\overline{f}(x,y)$ can be
expressed as in the following factorization 
in terms of the formal Puiseux series
\begin{equation}\label{eqn:2.13}
\overline{f}(t^N,y)=
\overline{u}(t^N,y) t^{N m_0} 
\prod_{j=1}^r
(y-\overline{\phi}_j(t))^{m_j},
\end{equation}
where $N$ is a positive integer, 
$m_0$ is a nonnegative integer,
$m_j$ are positive integers,   
$\overline{u}(x,y)\in\C [[x,y]]$ has 
a non-zero constant term
and 
$\overline{\phi}_j(t)\in \C [[t]]$ are distinct
(i.e., 
$\overline{\phi}_j(t)\neq \overline{\phi}_k(t)$
if $j\neq k$).
Let ${\mathcal R}(f)$ be the subset of 
$\{0,1,\ldots,r\}$ defined by  
\begin{equation}\label{eqn:2.14}
j\in{\mathcal R}(f)
\,\, \Longleftrightarrow \,\,
j=0 \,\, \mbox{ or } \,\,
\overline{\phi}_j(t)\in \R [[t]].
\end{equation}
The case $r=0$ is possible; 
when $\overline{f}(x,y)$ is expressed as
$\overline{u}(x,y) x^{m_0}$, 
we set ${\mathcal R}(f)=\{0\}$.
The quantity $\mu_0(f)$ is defined by 
\begin{equation}\label{eqn:2.15}
\mu_0(f)=\max\{m_j:j\in{\mathcal R}(f)\}.
\end{equation}
%
\begin{remark}
(1)\quad
It is obvious from the definition 
that the quantity $\mu_0(f)$ is determined by the
formal Taylor series of $f$ only, as well as 
the height $\delta_0(f)$ in (\ref{eqn:2.4}). 

(2)\quad 
We define $\mu_0(f)$ for a $C^{\infty}$ function
$f$ not satisfying the conditions (\ref{eqn:2.1}) as follows.
When $f(0,0)\neq 0$, 
${\mathcal R}(f)=\{0\}$ with $m_0=0$, 
which gives $\mu_0(f)=0$.
When $f(0,0)=0$ and $\nabla f(0,0)\neq 0$,
${\mathcal R}(f)=\{0\}$ with $m_0=1$ or
${\mathcal R}(f)=\{0,1\}$ with $m_0=0$ and $m_1=1$, 
which gives $\mu_0(f)=1$, by the implicit function theorem. 

(3)\quad The quantity $\mu_0(f)$ is invariant under 
the change of coordinates, which will be shown in Section~8.3.

(4)\quad 
When $f$ is real analytic and $\mu_0(f)\geq 1$, 
$\mu_0(f)$ is equal to the maximal order of vanishing 
of $f$ along the set 
$\{(x,y) \in \R^2: |x|^2+|y|^2=\gamma\}$
with sufficiently small $\gamma>0$ (see \cite{IkM11tams}). 

(5)\quad 
If a real analytic function $f$ satisfies  
$f(x,y)>0$ away from the origin, 
then $\mu_0(f)=0$ holds. 
But, in the $C^{\infty}$ case, the above implication is not true. 
For example, consider the $C^{\infty}$ function 
$f(x,y)=y^{2k}+ e^{-1/x^2}$ with $k\in\N$.  
In this case, $\mu_0(f)=2k$.
\end{remark}

More detailed properties of $\mu_0(f)$ will be 
investigated in Section~8.


\subsection{Main theorem}

Now let us state a main theorem in this paper, which
gives an answer to Problem~2.1. 
Indeed, we show that the meromorphically extendible region 
can be described by using the quantity $\mu_0(f)$.

\begin{theorem}
Let $f$ be a non-flat $C^{\infty}$ function defined in a neighborhood of 
the origin in $\R^2$. 
Then we have
\begin{enumerate}
\item 
If $\mu_0(f)=0,1$, then 
${\mathfrak m}_0(f)= \infty$ holds; 
\item
If $\mu_0(f)\geq 2$, then 
${\mathfrak m}_0(f)\geq 1/\mu_0(f)$ holds. 
\end{enumerate}
Furthermore,  when $\mu_0(f)<\delta_0(f)$, 
the poles of the extended local zeta function on 
${\rm Re}(s)>-1/\mu_0(f)$ exist 
in the finitely many arithmetic progressions 
that are constructed from negative rational numbers. 
\end{theorem}

\begin{figure}[H]

\begin{tikzpicture}
   \fill[gray!60!] (-0.4,-4.5) -- (-0.4,4.5) -- (4.5,4.5) -- (4.5,-4.5);
 \fill[gray!30!] (-0.4,-4.5) -- (-0.4,4.5) -- (-3,4.5) -- (-3,-4.5);
   \draw [thick, -stealth](-5,0)--(5,0) node [anchor=north]{Re};
   \draw [thick, -stealth](0,-5)--(0,5) node [anchor=west]{Im};
   \node [anchor=north west] at (0,0) {O};
   \node [anchor=north west] at (0,0) {O};
\coordinate (P2) at (-0.8,0);
\coordinate (P3) at (-1.2,0);
\coordinate (P4) at (-1.6,0);
\coordinate (P5) at (-2.0,0);
\coordinate (P6) at (-2.4,0);
\coordinate (P7) at (-2.8,0);
\node [anchor=north east] at (-0.3,0) {{\tiny $-\frac{1}{\delta_0(f)}$}};
\node [anchor=north east] at (-3,0) {{\tiny $-\frac{1}{\mu_0(f)}$}};
\filldraw[fill=gray!80!] (P2) circle[radius=0.8mm];
\filldraw[fill=gray!80!] (P3) circle[radius=0.8mm];
\filldraw[fill=gray!80!] (P4) circle[radius=0.8mm];
\filldraw[fill=gray!80!] (P5) circle[radius=0.8mm];
\filldraw[fill=gray!80!] (P6) circle[radius=0.8mm];
\filldraw[fill=gray!80!] (P7) circle[radius=0.8mm];
   \draw [dashed] (-3,-4.5)--(-3,4.5); 
\node [anchor=south east] 
at (-2.3,4.5) {{\tiny ${\rm Re}(s)=-\frac{1}{\mu_0(f)}$}};
   \draw [dashed] (-0.4,-4.5)--(-0.4,4.5); 
\node [anchor=south east] at (0,4.5) {{\tiny ${\rm Re}(s)=-\frac{1}{\delta_0(f)}$}};

\coordinate (P1) at (-0.4,0);
\filldraw[fill=gray!80!] (P1) circle[radius=0.8mm]; 

\node [above] at (-1.8,0.2) {{\small Poles}};

\draw [->, color=black] (-0.8,2)--(-2.3,2);
\draw [->, color=black] (-0.8,3.5)--(-2.3,3.5);
\draw [->, color=black] (-0.8,-2)--(-2.3,-2);
\draw [->, color=black] (-0.8,-3.5)--(-2.3,-3.5);
\end{tikzpicture}

\caption[]{In the case where $\mu_0(f)<\delta_0(f)$.}
\end{figure}
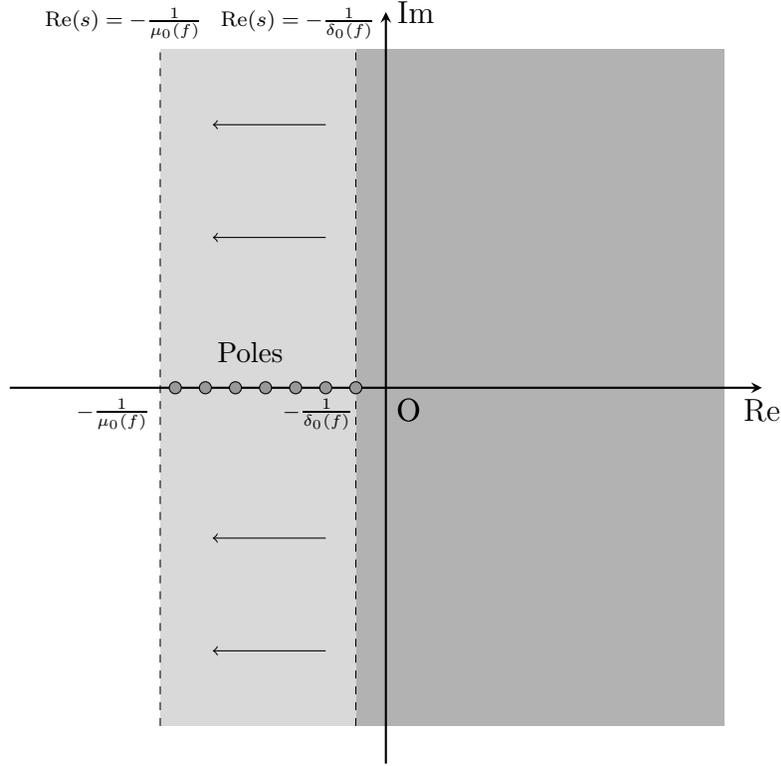

\begin{remark}
(1)\quad
The assumption of the theorem 
does not need the condition (\ref{eqn:2.1}).

(2) \quad 
Recalling Theorem~2.1 given by Greenblatt \cite{Gre06}, 
we can see
$\mu_0(f)\leq \delta_0(f)$ for $f\in C^{\infty}(U)$ by using 
the above theorem with 
the obvious inequality ${\mathfrak m}_0(f)\leq {\mathfrak h}_0(f)$. 
The inequality $\mu_0(f)\leq \delta_0(f)$ itself
will be directly shown in Section~8.4.

(3)\quad
Since the equality ${\mathfrak m}_0(f)=1/\mu_0(f)$ holds 
for $f$ in (\ref{eqn:2.11}), 
the estimate in (ii) is optimal in the uniform sense for $f$.
From the obvious inclusion $C^{\omega}(U)\subset C^{\infty}(U)$, 
there are many $C^{\infty}$ functions $f$ such that
$\mu_0(f)\geq 2$ and ${\mathfrak m}_0(f)=\infty$
(in particular,  ${\mathfrak m}_0(f)>1/\mu_0(f)$).  
The optimality of the estimate in (ii)
will be more precisely discussed in Section~14.2. 

(4)\quad 
At present, very few properties 
of non-polar singularities of local zeta functions 
are known. 
We will also discuss these issues in Section~14. 

\end{remark}



\section{Almost desingularization theorem for $C^{\infty}$ functions}

In the discussion below in Sections~3-7, 
local properties of every function are essentially important and
we do not care how small the domain of definition
of each function is. 
Thus, it will be convenient to formulate our results
for {\it function-germs} rather than functions.
An identity involving several function-germs is defined 
to be true if there exist functions from the equivalence
classes of these germs such that in the intersection 
of their domains of definition the identity is true
in the usual sense. 
We will make use of the following rings of germs
of complex-valued function in Sections 3--7 and 14:
\begin{itemize}
\item $C((x))$ 
------ 
the set of germs of continuous functions at the origin of $\R$.
\item $C^{\infty}((x))$ and $C^{\infty}((x,y))$ 
------
the rings of germs of 
$C^{\infty}$ functions at the origin of $\R$ and $\R^2$,
respectively. 
\end{itemize}
The rings of germs of real-valued functions will be
denoted by adding an $\R$ to the above notation, 
e.g. $\R C^{\infty}((x,y))$.

For a given ring $R$, 
an element of $R$ which has an inverse is 
called a unit. 
An element of 
the ring of the formal power series 
is a unit if and only if 
it has non-zero constant term. 


In Sections 3--7, we always assume that 
$F(x,y)\in \R C^{\infty}((x,y))$  
satisfies that
its formal Taylor series 
can be expressed as in a factorization form  
\begin{equation}\label{eqn:3.1}
\overline{F}(x,y)=
\overline{u}(x,y)
\prod_{j=1}^r (y-\overline{\Phi}_j(x))^{m_j},
\end{equation}
where 
$m_j$ are positive ingegers, 
$\overline{\Phi}_j(x)\in\C[[x]]$ are distinct
(i.e., $\overline{\Phi}_j(x)\neq\overline{\Phi}_k(x)$ 
if $j\neq k$) and
$\overline{u}(x,y)\in \R[[x,y]]$ 
is a unit. 
Let $n:=\sum_{j=1}^r m_j$.
Corresponding to the general case (\ref{eqn:2.13}), 
we now give the additional assumptions: 
$m_0=0$ and $N=1$.
However, these assumptions do not essentially restrict
any properties of $C^{\infty}$ functions dealt with
in our analysis of local zeta functions
(see Sections 7--11).  

Since $F$ is a $C^{\infty}$ function,  
it is impossible to resolve the singularities of
the zero variety of $F$ in general by using 
algebraic transforms only. 
However, we attempt to do so as much as possible  
by using a composition of a finite number of 
blowings up, 
which will be explained in Section~4.
As a result, 
we succeed in giving an ``almost'' resolution of singularities
of the zero variety, 
which locally expresses $F(x,y)$ 
in the ``almost'' normal crossings form. 
The exact meaning of ``almost''  is as follows.


\begin{definition}
Let $f(x,y)\in C^{\infty}((x,y))$. 
\begin{enumerate} 
\item[(1)]
$f(x,y)$ is said to be expressed 
in the {\it normal crossings} form
if it is locally expressed as 
\begin{equation*}
f(x,y)=u(x,y)x^a y^m,
\end{equation*}
where $a, m$ are nonnegative integers and 
$u(x,y)\in C^{\infty}((x,y))$ satisfies $u(0,0)\neq 0$. 
\item[(2)] 
$f(x,y)$ is said to be expressed in 
the {\it almost normal crossings} form
if it is locally expressed as 
\begin{equation}\label{eqn:3.2}
f(x,y)=u(x,y)x^a 
\left(
y^m+\varepsilon_1(x) y^{m-1}+\cdots+\varepsilon_m(x)
\right),
\end{equation}
where $a, m$ are nonnegative integers, 
$u(x,y)\in C^{\infty}((x,y))$ satisfies $u(0,0)\neq 0$ and
$\varepsilon_j(x)\in C^{\infty}((x))$ are flat at the origin.
\end{enumerate} 
\end{definition}

Considering the case where $\varepsilon_j\equiv 0$ for all $j$, 
we see that the concept of ``normal crossings" is a special case of 
the concept of ``almost normal crossings". 
This subtle difference gives a serious influence in the 
analytic continuation of local zeta functions.


\subsection{Almost resolution of singularities}

The following theorem is the most important 
result in this paper from a geometrical point of view.
After preparing many kinds of tools in Sections 4--6,  
we will give a proof of this theorem in Section~7. 

Recall ${\mathcal R}(F)=\{j:\overline{\Phi}_j(x)\in\R[[x]]\}$,
where $\overline{\Phi}_j(x)$ is as in (\ref{eqn:3.1}) 
(see Section~2.4).

\begin{theorem}
Let $F(x,y)$ be a real-valued $C^{\infty}$ function 
defined near the origin in $\R^2$.
If 
$F(x,y)$ satisfies that 
its Taylor series admits the factorization (\ref{eqn:3.1}), then 
there exist an open neighborhood $U$ of the origin in $\R^2$, 
a two-dimensional $C^{\infty}$ real manifold $Y$ and 
a proper map $\pi$ from $Y$ to $U$ such that 
\begin{enumerate}
\item 
$\pi$ is a local diffeomorphism from 
$Y - \pi^{-1}(O)$ to 
$U - \{O\}$; 
\item
For each $j\in{\mathcal R}(F)$, there exist a point 
$P_j$ on $\pi^{-1}(O)$ and 
a local $C^{\infty}$ coordinate $(x,y)$ 
centered at $P_j$ so that
the following (a), (b) hold:
\begin{enumerate}
\item
$(F\circ \pi)(x,y)$ can be locally expressed 
in the almost normal 
crossings form. 
To be more specific,  
\begin{equation}\label{eqn:3.3}
(F\circ \pi) (x,y)=
u_j(x,y)
x^{a_j} 
\left(
y^{m_j}+\varepsilon_{j1}(x)y^{m_j-1}+\cdots+\varepsilon_{jm_j}(x)
\right), 
\end{equation}
where 
$a_j$ is a nonnegative integer, 
$m_j\in\N$ is as in (\ref{eqn:3.1}), 
$\varepsilon_{jk}(x)\in \R C^{\infty}((x))$ 
are flat functions at the origin and 
$u_j(x,y)\in \R C^{\infty}((x,y))$ satisfies 
$u_j(0,0)\neq 0$;
\item
The Jacobian of $\pi$ is locally expressed as 
\begin{equation}\label{eqn:3.4}
J_{\pi}(x,y)=x^{M_j},
\end{equation}
where $M_j$ is a nonnegative integer; 
\end{enumerate}
\item For each 
$Q\in \pi^{-1}(O)-\{P_j:j\in{\mathcal R}(F)\}$,
there exists a
local $C^{\infty}$ coordinate $(x,y)$ centered at $Q$ so that
the following locally hold:
\begin{equation}\label{eqn:3.5}
(F\circ \pi) (x,y)= u_Q(x,y) x^{A_Q} y^{B_Q} 
\quad \mbox{and} \quad 
J_{\pi}(x,y)=x^{C_Q},
\end{equation}
where $u_Q(x,y)\in \R C^{\infty}((x,y))$ satisfies
$u_Q(0,0)\neq 0$ and 
$A_Q, B_Q, C_Q$ are nonnegative integers. 
\end{enumerate}
\end{theorem}

\begin{remark}
Let us consider the case where the 
singularities of the zero variety of $F$ at the origin 
can be completely resolved; i.e., 
$F\circ \pi$ can be locally expressed in the normal crossings form
at {\it any} point on $\pi^{-1}(O)$.
(In this case, every local zeta function always admits 
the meromorphic extension 
to the whole complex plane.)

\begin{enumerate}
\item
When $F$ is real analytic, 
the functions $\varepsilon_{jk}(x)$ in (\ref{eqn:3.3}) must be
identically zero.
Therefore, $F\circ\pi$ can be locally expressed
in the normal crossings form 
at any point on $\pi^{-1}(O)$, which
is a particular version of the desingularization 
theorem of Hironaka (see \cite{Ati70}).
\item
In the case of $\mu_0(F)=0$, 
the case (ii) does not occur.
(In this case, Theorem~3.2 is the same as Proposition~6.14, below.) 
\item
In the case of $\mu_0(F)=1$,  
$F\circ\pi$ in (\ref{eqn:3.3}) can be also
expressed in the normal crossings form $u(x,y)x^a y$
after a slight change of local coordinates. 
\end{enumerate}
Note that some cases in (ii), (iii) 
satasfy the $\R$-nondegeneracy condition
in the sense of Kouchnirenko 
(\cite{AGV88}, see also Section~8.4 in this paper).
It has been shown in \cite{KaN16jmst} that 
(toric) resolution of singularities can be constructed
under the $\R$-nondegeneracy condition and 
the meromorphic continuation of local zeta functions
can be precisely understood. 
\end{remark}

In the general $C^{\infty}$ case, 
the above theorem with its proof shows that 
$F\circ\pi$ can be locally expressed in the ``almost''
normal crossings form at any point on 
$\pi^{-1}(O)$ by using a composition of a finite number of 
blowings up.

\begin{definition}
The proper map $\pi:Y\to U$ 
in Theorem~3.2 is called 
an {\it almost resolution of singularities}
for a $C^{\infty}$ function $F$ at the origin. 
\end{definition}

Notice that almost resolution of singularities 
does not always resolve the singularities of the zero varieties
of $C^{\infty}$ functions. 
This phenomenon can be only found in the $C^{\infty}$ case 
and it may be interpreted as
an essential difference between the geometric 
properties of the zero varieties of real analytic functions 
and $C^{\infty}$ functions.
From an analytical point of view, 
as is shown by the example (\ref{eqn:2.11}) in Section~2,  
non-zero flat functions may
give an obstruction for 
the meromorphic extension of 
local zeta functions in the $C^{\infty}$ case.   
On the other hand, 
as shown in Theorem~2.2,  
the existence of flat functions gives 
no influence on the determination of 
${\mathfrak h}_0(f)$.


\section{Construction of blowings up}


In this section, we recall ordinary {\it blowing up} 
which is an important tool in the studies of algebraic geometry
and so on.

After constructing an appropriate complex manifold by using 
blowings up, 
we obtain a desired real manifold 
by restricting this complex manifold to the real space. 
From this reason, the choice of local coordinates must be 
sufficiently cared. 

\subsection{Blowing up of an open set in $\C^2$}

Let $P=(a,b)$ be a point on $\C^2$ 
and let $U$ be an open neighborhood of $P$. 
Let us recall a blowing up of $U$ with center $P$. 

Let $X_1$ be the subset of ${\mathbb P}^1(\C)\times U$
defined by
\begin{equation}\label{eqn:4.1}
X_1:=\{
((c_0:c_1),(x,y))\in {\mathbb P}^1(\C)\times U:
c_1 (x-a)=c_0 (y-b)
\}
\end{equation}
and let $\sigma_1:X_1\to U$ be a projection defined by
\begin{equation}\label{eqn:4.2}
\sigma_1((c_0:c_1),(x,y))=(x,y).
\end{equation}
Then it is well known (c.f. \cite{Oka10}) that 
\begin{enumerate}
\item 
$X_1$ is a two-dimensional complex manifold 
and 
$\sigma_1$ is a proper map; 
\item
$\sigma_1:X_1 - \sigma_1^{-1}(P) \to U - \{P\}$
is isomorphism;
\item
The {\it exceptional curve} 
$E_0:=\sigma_1^{-1}(P)$ is isomorphic to 
${\mathbb P}^1(\C)$ as a complex manifold. 
\end{enumerate}
The above proper map $\sigma_1:X_1\to U$
is called a {\it blowing up of} $U$ {\it with center} $P$.
The complex structure of $X_1$ is specified by 
the set of local charts $\{(V_j,\varphi_j)\}_{j=0,1}$, 
where 
\begin{equation}\label{eqn:4.3}
V_j:=\{((c_0:c_1),(x,y))\in X_1:
c_j\neq 0 \}
\end{equation} 
and $\varphi_j:V_j \to \varphi_j(V_j)=:\tilde{V}_j$, 
for $j=0,1$, is defined by 
\begin{equation}\label{eqn:4.4}
\begin{split}
&\varphi^{-1}_0(u,v)=((1:v),(u+a,uv+b)), \\
&\varphi^{-1}_1(z,w)=((z:1),(zw+a,w+b)). 
\end{split}
\end{equation}
We call  $(u,v)$ (resp. $(z,w)$) 
the {\it canonical coordinate} on $V_0$ (resp. on $V_1$).
From (\ref{eqn:4.4}), a coordinate transformation 
$\varphi_1\circ\varphi_0^{-1}:
\tilde{V}_0-\{v=0\} \to \tilde{V}_1-\{z=0\}$
is expressed as 
\begin{equation}\label{eqn:4.5}
(\varphi_1\circ\varphi_0^{-1})(u,v)(=(z,w))=(1/v,uv).
\end{equation}
For example, when 
$U=\{(x,y)\in \C^2:|x-a|<\delta, |y-b|<\delta\}$ 
with $\delta>0$, 
$X_1$ is constructed by piecing together the following
open subsets of $\C^2$ by using (\ref{eqn:4.5}):
\begin{equation*}
\begin{split}
&\tilde{V}_0:=
\{(u,v)\in\C^2:|u|<\delta, |uv|<\delta \}, \\
&\tilde{V}_1:=
\{(z,w)\in\C^2:|zw|<\delta, |w|<\delta\}.
\end{split}
\end{equation*}

\subsection{A series of blowings up}
Next, let us construct a series of blowings up. 
We say that a two-dimensional complex manifold 
$X$ satisfies {\it Property} $(T)$, 
if a complex structure of $X$ 
is given by a set of local charts 
$\{(U_j,\varphi_j)\}_{j}$ 
satisfying the condition:
if $U_j\cap U_k\neq \emptyset$, then
the coordinate transformation 
$\varphi_k\circ\varphi_j^{-1}:
\varphi_j(U_j\cap U_k)\to 
\varphi_k(U_j\cap U_k)$ is a biholomorphic map 
for $j,k=0,\ldots,n$, which is 
defined by a finite composition of maps of the forms
\begin{equation}\label{eqn:4.6}
\begin{split}
&(u,v)\to (v,u), \quad
(u,v) \to (u,uv), \\
&(u,v) \to (1/v, uv), \quad 
(u,v) \to (u+a,v+b), \quad (a,b\in\C).
\end{split}
\end{equation}
In other words, the complex manifold $X$ 
with Property $(T)$  
is constructed by piecing together open sets 
$\varphi_j(U_j)\subset\C^2$ 
via the above maps (\ref{eqn:4.6}).
It is easy to see that 
the complex manifold $X_1$ in (\ref{eqn:4.1}) 
satisfies Property $(T)$. 

Let $X_n$ be a complex manifold with local charts 
$\{(U_j,\varphi_j)\}_{j=0}^l$ 
having Property $(T)$. 
Let $P_n$ be a point on $X_{n}$. 
Now let us define a blowing up of $X_{n}$ 
with center $P_n$.
There exists a local chart containing $P_n$, 
which may be $U_0$. 
Let $U_* (\subset U_0)$ be an open neighborhood of $P_n$ 
and denote $\tilde{U}_*:=\varphi_0(U_*)$ and 
$\tilde{P}_n:=\varphi_0(P_n)$.
We obtain
$\tilde{\sigma}:\tilde{X}\to \tilde{U}$, 
which is a blowing up of $\tilde{U}$ 
with center $\tilde{P}_n$, 
by the same way as in (\ref{eqn:4.1}), (\ref{eqn:4.2}). 
A new complex manifold $X_{n+1}$ 
is constructed by piecing 
together $X_{n}-\{P_n\}$ and 
$\tilde{X}$ 
using the equivalence of $U_*-\{P_n\}$ and 
$\tilde{\sigma}^{-1}(\tilde{U}_* - \{\tilde{P}_n\})$
via the equivalence of each with 
$\tilde{U}_*-\{\tilde{P}_n\}$. 

Let $\{(V_k,\phi_k)\}_{k=0,1}$ be the set of local 
charts of $\tilde{X}$ given by the same way
as in (\ref{eqn:4.3}), (\ref{eqn:4.4}). 
The complex structure of $X_{n+1}$ 
is specified by a set of local charts   
$$
\{(U_0-\{P_n\}, \varphi_0^*)\} \cup 
\{(U_j,\varphi_j)\}_{j=1}^l \cup
\{(V_k,\phi_k)\}_{k=0,1},
$$ 
where $\varphi_0^*$ is the restriction 
of $\varphi_0$ to $U_0-\{P_n\}$.
The {\it canonical coordinate} can be inductively 
introduced on each local chart by using (\ref{eqn:4.5}).
Moreover, 
it is easy to check that each coordinate transformation
is expressed by using finite composition of the maps 
in (\ref{eqn:4.6}). 
Therefore, $X_{n+1}$ also has Property ($T$).

The proper map $\sigma_{n+1}:X_{n+1}\to X_{n}$ 
is defined as follows. 
The restriction of  
$\sigma_{n+1}$ to $V_j$ is decided by 
$\varphi_{0}^{-1}\circ\tilde{\sigma}$ for $j=0,1$
and that of $\sigma_{n+1}$
to $\left(\bigcup_{j=1}^{n} U_j\right) 
\cup (U_0-\{P_n\})$ 
is the identity map.
This is well-defined and 
the map $\sigma_{n+1}:X_{n+1}\to X_{n}$ 
is called a {\it blowing up of} $X_{n}$ 
{\it with center} $P_n$.

From the above inductive process, 
a series of blowings up 
\begin{equation}\label{eqn:4.7}
\begin{CD}
\cdots  @>{\sigma_{n+1}}>>   X_{n} @>{\sigma_{n}}>> X_{n-1}  
@>{\sigma_{n-1}}>> \cdots
@>{\sigma_{2}}>> X_1
@>{\sigma_{1}}>> X_0:=U\subset\C^2.
  \end{CD}
\end{equation}
is constructed. 
%
%
Here, for each $n\in\N$, 
let $\sigma_{n}:X_{n}\to X_{n-1}$ 
be the blowing up of $X_{n-1}$ with center
a point on $X_{n-1}$ and 
$U$ is an open neighborhood of the origin in $\C^2$.  
We call $E_{n-1}=\sigma_{n}^{-1}(P_{n-1})$  
the {\it exceptional curve} of $\sigma_n$ 
for $n\in\N$. 
For $n\in\N$, the composition of the 
blowings up in (\ref{eqn:4.7}) is written as 
\begin{equation}\label{eqn:4.8}
\pi_n:=
\sigma_{1}\circ\sigma_{2}\circ\cdots
\circ \sigma_n: X_{n}\to X_0.
\end{equation}
%

When a series of blowings up is actually constructed
to provide an appropriate resolution of singularities
for our purpose
(see the proof of Proposition~6.15), 
each center is chosen as follows.
\begin{definition}
We say that 
(\ref{eqn:4.7}) is a series of blowings up {\it with real centers}, 
if each center $P_n$ is $(a_n,b_n)$ with $a_n,b_n\in\R$
on a canonical coordinate.
\end{definition}


\subsection{A series of real blowings up}

We say that a two-dimensional $C^{\omega}$ real manifold 
$Y$ has {\it property} $(T_{\R})$, 
if $Y$ admits a set of local charts 
$\{(U_j,\varphi_j)\}_{j}$ 
satisfying the condition:
if $U_j\cap U_k\neq \emptyset$, then
the coordinate transformation 
$\varphi_k\circ\varphi_j^{-1}:
\varphi_j(U_j\cap U_k)\to 
\varphi_k(U_j\cap U_k)$ is an isomorphism 
for $j,k=0,\ldots,n$, which is 
defined by a finite composition of maps in (\ref{eqn:4.6})
with $u,v, a,b\in\R$. 

By noticing the form of maps in (\ref{eqn:4.6}) 
with $a,b\in\R$, 
we can define 
a map in the real version analogous to the blowings up 
in a similar fashion to the above procedure. 
We denote this map by 
$\hat{\sigma}_{n+1}:Y_{n+1}\to Y_{n}$ for $n\in\Z_+$, 
where $Y_{n+1}$, $Y_{n}$ are 
two-dimensional $C^{\omega}$ real manifolds 
having Property ($T_{\R}$), 
which is called a
{\it real blowings up of} $Y_n$ {\it with center} $P_n \in Y_n$.
Furthermore, a series of real blowings up can be 
inductively constructed as
\begin{equation}\label{eqn:4.9}
\begin{CD}
  \cdots @>{\hat{\sigma}_{n+1}}>>  
Y_{n} @>{\hat{\sigma}_{n}}>> Y_{n-1}  
@>{\hat{\sigma}_{n-1}}>> \cdots
@>{\hat{\sigma}_{2}}>> Y_1
@>{\hat{\sigma}_{1}}>> Y_0=:U\subset\R^2,
  \end{CD}
\end{equation}
where $U$ is an open neighborhood of the origin in $\R^2$.
For $n\in\N$, the composition of 
real blowings up is written as 
\begin{equation}\label{eqn:4.10}
\hat{\pi}_n:=
\hat{\sigma}_{1}\circ\hat{\sigma}_{2}\circ\cdots
\circ\hat{\sigma}_n: Y_{n}\to Y_0.
\end{equation}

When (\ref{eqn:4.7}) is 
a series of blowings up with real centers, 
that of real blowings up (\ref{eqn:4.9}) 
can be simultaneously defined.  
Then, 
$\hat{\sigma}_n$ may be regarded as  
the restriction of $\sigma_n$ to $Y_n$ and 
the real manifold 
$Y_n$ can be regarded as a natural embedding in 
the complex manifold $X_n$ for 
$n\in\Z_+$. 
In this case, we write $Y_n=\hat{X}_n$. 


\subsection{Geometry of exceptional curves}

Let us consider how the exceptional curves 
are transformed by a series of blowings up (\ref{eqn:4.7}).  

Let $E_0$ be the exceptional curve of $\sigma_1$ in $X_1$. 
Let $\sigma_2:X_2 \to X_1$ be a blowing up 
with center a point $P_1$ in $X_1$. 
The closure of $\sigma_1^{-1}(E_0-\{P_1\})$ in $X_2$ 
will be denoted by $E_0$ in the same manner. 
We call $E_0,E_1$ in $X_2$ the 
exceptional curves of the composition map 
$\pi_2=\sigma_2 \circ \sigma_1$. 
Repeating this process inductively, 
we can define the {\it exceptional curves} $E_0,E_1,\ldots, E_{n-1}$ of 
{\it the composition map} $\pi_n$ in (\ref{eqn:4.8}). 
Note that $\pi_n^{-1}(0)=\bigcup_{j=0}^{n-1} E_j$.
The set of the exceptional curves of $\pi_n$ in $X_n$ 
is denoted by ${\mathcal E}^{(n)}$. 
The following lemma shows 
the geometrical situation of the exceptional curves of $\pi_n$ 

\begin{lemma}[\cite{Wal04}, Proposition~3.4.3, p.47]
Let $n\in\N$.
The exceptional curve $E_n$ in $X_{n+1}$ 
intersects $E_{n-1}$ and at most one curve 
$E_j$ with $j<n-1$. 
These intersections are transverse, and no
three of the exceptional curves pass through a 
common point. 
\end{lemma}

Analogously,  
we can also define the {\it real exceptional curves} 
$\hat{E}_0, \hat{E}_1,\ldots, \hat{E}_{n-1}$ of 
the composition map $\hat{\pi}_n$ in $Y_n$. 
Moreover, when $E_j$ and $X_{n+1}$ are replaced by 
$\hat{E}_j$ and $Y_{n+1}$, 
the assertion in Lemma~4.2 also holds.


\section{Rychkov's factorization formula and decisive curves}

In this section, let $F(x,y)\in\R C^{\infty}((x,y))$
satisfy that its formal Taylor series admits 
the factorization (\ref{eqn:3.1}) in Section~3. 
Let us explain the important factorization formula 
of Rychkov \cite{Ryc01}, 
which clarifies geometrical properties of the zero variety 
of $F(x,y)$. 

\begin{proposition}[\cite{Ryc01}]
For each $j=1,\ldots,r$, 
there exist a 
$\Phi_{j}(x)\in C^{\infty}((x))$
and $\gamma_{jk}(x) \in C((x))$
for $k=1,\ldots,m_j$ 
such that 
\begin{enumerate}
\item $\Phi_j(x)$ admits the formal 
Taylor series $\overline{\Phi}_j(x)$;
\item $\gamma_{jk}(x)=O(x^l)$ as $x\to 0$ 
for any $l\in\N$, $k=1,\ldots,m_j$;
\item $F(x,\Phi_{j}(x)+\gamma_{jk}(x))=0$ 
for $x\in(-\delta,\delta)$
with small $\delta>0$, 
$k=1,\ldots,m_j$;
\item
For each $\alpha\in\N$, 
let ${\mathcal E}_{j\alpha}(x)$ 
be a continuous function defined by 
\begin{equation}\label{eqn:5.1}
{\mathcal E}_{j\alpha}(x):=\sum_{k=1}^{m_j}
[\gamma_{jk}(x)]^{\alpha}.
\end{equation}
Then ${\mathcal E}_{j\alpha}(x)$ belong to 
$C^{\infty}((x))$ for any $\alpha\in\N$. 
\item
If we additionally assume that 
$\overline{\Phi}_j(x)$ belongs to $\R[[x]]$, then 
$\Phi_j(x)$ and ${\mathcal E}_{j\alpha}(x)$ 
belong to $\R C^{\infty}((x))$
for any $\alpha\in\N$.
\end{enumerate}
\end{proposition}

For $j=1,\ldots,r$, we define
\begin{equation}\label{eqn:5.2}
F_j(x,y):= \prod_{k=1}^{m_j}
(y-\Phi_{j}(x)-\gamma_{jk}(x)).
\end{equation}
Although the continuous functions $\gamma_{jk}(x)$ may not
have the $C^{\infty}$ differentiable property,  
all the elementary symmetric polynomials  
in the variables $\gamma_{j1}(x),\ldots,\gamma_{jm_j}(x)$
belong to $C^{\infty}((x))$ from (iv). 
Therefore, $F_j(x,y)$ can be rewritten as 
\begin{equation}\label{eqn:5.3}
F_j(x,y)=y^{m_j}+p_{j1}(x)y^{m_j-1}+\cdots+
p_{jm_j}(x),
\end{equation}
where $p_{jk}(x)\in C^{\infty}((x))$ for $k=1,\ldots,m_j$
(i.e., $F_j(x,y)\in C^{\infty}((x))[y]$). 
Furthermore, under the assumption 
$\overline{\Phi}_j(x)\in\R[[x]]$, 
$F_{j}(x,y)$ takes the same form as (\ref{eqn:5.3})
where $p_{jk}(x)\in \R C^{\infty}((x))$ 
for $k=1,\ldots,m_j$
(i.e., $F_j(x,y)\in \R C^{\infty}((x))[y]$)
from the above (v)). 


It follows from the Malgrange preparation theorem 
(c.f. \cite{GoG73}, p.95)
that $F(x,y)$ can be expressed as 
\begin{equation}\label{eqn:5.4}
F(x,y)=u(x,y)\left(
y^n+a_1(x) y^{n-1}+\cdots + a_n(x)
\right),
\end{equation}
where 
$u(x,y)\in \R C^{\infty}((x,y))$ satisfies $u(0,0)\neq 0$ and
$a_j(x)\in \R C^{\infty}((x))$ satisfy 
$a_j(0)=0$.
Since $\sum_{j=1}^r m_j=n$, 
$F(x,y)$ can be factorized as
\begin{equation}\label{eqn:5.5}
F(x,y)=
u(x,y) \prod_{j=1}^r
F_j(x,y)=
u(x,y) \prod_{j=1}^r
\prod_{k=1}^{m_j}
(y-\Phi_{j}(x)-\gamma_{jk}(x)),
\end{equation}
from (\ref{eqn:5.2}), (\ref{eqn:5.4}).  

Let $B_{jk}$ be the set in $\R\times\C$ 
locally defined near the origin by the parametrization
\begin{equation}\label{eqn:5.6}
x=t,\quad y=\Phi_{j}(t)+\gamma_{jk}(t)
\quad\quad  \mbox{for $t\in \R$}.
\end{equation}
The union of all $B_{jk}$ for all $j,k$ 
is called the {\it decisive curve} 
defined by $F$, 
which is denoted by $C_F$.
Each $B_{jk}$ is called a {\it branch} of 
the decisive curve $C_F$. 
Since the notions of the decisive curve and its branches 
are locally defined near the origin in $\C^2$,
they can be naturally defined at any point
on the two-dimensional complex manifolds. 

A branch $B_{jk}$ defined by (\ref{eqn:5.6}) is called 
a {\it real branch} if 
the formal Taylor series $\overline{\Phi}_j(t)$ 
of $\Phi_j(t)$
belongs to $\R[[t]]$, 
otherwise it is called a {\it non-real branch}.  
We remark that
if a branch is contained in $\R^2$ near the origin, 
then it is a real branch; while 
the converse is not always true in the 
$C^{\infty}$ setting. 
For example, consider the $C^{\infty}$ function
$F(x,y)=y^2+e^{-2/x^2}$. 
In this case, there are  two branches defined by 
$x=t$,  $y=\pm i e^{-1/t^2}$ for $t\in\R$.
They are not contained in $\R^2$ but 
they are real branches. 

\begin{remark}
For general $C^{\infty}$ functions $f$ with $f(0,0)=0$, 
the decisive curves $C_f$ can be similarly defined
by using the factorization of Rychkov \cite{Ryc01}.
(This generalization is not necessary in the analysis
of local zeta functions below.)  
\end{remark}


\section{The transforms of a decisive curve via blowings up}

In this section, we attempt to construct a series of blowings up 
in order to resolve the singularities of the decisive curve
defined by $F$.
In this section, 
we say that $\gamma(t)\in C((t))$ has a {\it flat property} if
$\gamma(t)=O(t^l)$ as $t\to 0$ for any $l\in\N$. 


\subsection{The transforms of a branch}

Let us carefully observe how a branch of the decisive curve 
is transformed by a series of blowings up. 


Let $U$ be an open neighborhood of the origin in $\C^2$. 
Let $B$ be a branch at the origin $O=:P_0$ of 
the decisive curve $C_F$ 
in $U$, which is locally expressed as
\begin{equation}\label{eqn:6.1}
x=t, \quad y=\Phi_0(t)+\gamma_0(t)
\quad\quad\mbox{ for $t\in\R$},
\end{equation}
where 
$\Phi_0(t)\in C^{\infty}((t))$ satisfies $\Phi_0(0)=0$ and
$\gamma_0(t)\in C((t))$ has a flat property.
We remark that $\gamma_0(t)$ may be a complex-valued function. 

Blowing up with center $P_0$ produces 
a complex manifold $X_1$ as in Section~4.1
and the exceptional curve $E_0=\sigma_1^{-1}(P_0)$. 
Let $B^{(1)}$ be the closure of 
$\sigma_1^{-1}(B-\{P_0\})$, 
which is called the {\it strict transform} of $B$. 
As explained in the construction of blowings up in Section~4, 
$X_1$ admits the set of local charts 
$\{(V_j,\varphi_j)\}_{j=0,1}$ as in (\ref{eqn:4.3}).  
Denote $\tilde{V}_j=\varphi_j(V_j) \subset \C^2$ 
for $j=1, 2$. 

First, let us observe geometrical situations of the strict transform  
$B^{(1)}$ and the exceptional curve $E_0$ on $V_0$.  
From the definition of blowing up, 
$\sigma_1$ can be regarded as 
the map from $\tilde{V}_0$ to $U$ 
given by
\begin{equation}\label{eqn:6.2}
(u,v)\mapsto (x,y)=(u,uv).
\end{equation} 
Then the Jacobian of $\pi_1$ satisfies 
$J_{\pi_1}(u,v)=u$ and 
the strict transform $B^{(1)}$ is locally expressed as 
\begin{equation}\label{eqn:6.3}
u=t, \quad v=\Phi_1(t)+\gamma_1(t)
\quad\quad\mbox{ for $t\in\R$},
\end{equation}
where $\Phi_1(t)=\Phi_0(t)/t\in C^{\infty}((t))$ and
$\gamma_1(t)=\gamma_0(t)/t\in C((t))$. 
We remark that 
$\gamma_1(t)$ also has a flat property.
The exceptional curve $E_0$ is expressed 
as $u=0$, $v=\tau$ for $\tau\in\C$ on $\tilde{V}_0$, and
the strict transform $B^{(1)}$ meets $E_0$
at a unique point $P_1=(0,\Phi_1(0))$, where 
$\Phi_1(0)=\lim_{t\to 0} \Phi(t)/t$.  
Note that $B^{(1)}$ transversely intersects 
$E_0$. 

Next, the geometrical situation of $B^{(1)}$ and 
$E_0$ on $V_1$ is as follows. 
From the definition of blowing up, 
$\sigma_1$ can be regarded as the map 
from $\tilde{V}_1$ to $U$ given by 
\begin{equation}\label{eqn:6.4}
(z,w) \mapsto (x,y)=(zw,w).
\end{equation}
The exceptional curve $E_0$ is expressed as 
$z=\tau$, $w=0$ for $\tau\in\C$ on $\tilde{V}_1$. 
It is easy to see that $B^{(1)}$ does not intersect 
$E_0$ on $\tilde{V}_1$. 

Inductively, let us assume that
a complex manifold $X_n$ with the set of local charts
$\{(U_j,\varphi_j)\}_{j=0}^l$
having Property ($T$) in Section~4.2
and that
there exists a local chart of $X_n$, which may be $U_0$, 
with the canonical coordinate, on which  
an exceptional curve $E_*$ of $\pi_{n-1}$ is 
expressed as $x=0$, $y=\tau$ for $\tau\in\C$ and 
$B^{(n)}$ is a subset of $X_n$ 
locally parametrized as 
\begin{equation}\label{eqn:6.5}
x=t, \quad y=\Phi_n(t)+\gamma_n(t)
\quad\quad\mbox{ for $t\in\R$},
\end{equation}
where 
$\Phi_n(t)\in C^{\infty}((t))$ and 
$\gamma_n(t)\in C((t))$ has a flat property. 
Note that $B^{(n)}$ meets $E_*$
at a unique point $P_n^*=(0,\Phi_n(0))$. 
Let $P_n$ be a point on $X_n$. 
A blowing up of $X_n$ with center $P_n$ gives 
a new complex manifold $X_{n+1}$
and a new map $\sigma_{n+1}:X_{n+1}\to X_n$ as explained in 
Section~4.3.
We write $E_*$ again for the closure of 
$\sigma_{n+1}^{-1}(E_*-\{P_n\})$ and 
$E_n$ for the exceptional curve of $\sigma_{n+1}$.
Let $B^{(n+1)}$ be the closure of 
$\sigma^{-1}_n(B^{(n)}-\{P_n\})$. 
For a branch $B=:B^{(0)}$ of $C_F$, 
we can inductively define 
$B^{(n)}$ on $X_n$ for each $n\in\N$, 
which is called 
the ($n$-th) {\it strict transform} of a branch $B$. 
Let $U$ ($\subset U_0$) be an open neighborhood 
of $P_n$ and 
let $\{(V_j,\phi_j)\}_{j=0,1}$ be the set of 
new local charts of $X_{n+1}$ 
produced in the blowing up process in Section~4.
Denote $\hat{V}_j:=\phi_j(V_j)\subset\C^2$
for $j=1,2$.
Here, $V_0$ (resp. $V_1$) 
admits the canonical coordinate $(u,v)$ in (\ref{eqn:4.4})
(resp. $(z,w)$ in (\ref{eqn:4.4})).

First, let us consider the case where the center is 
the point $P_n^*$. 
In a similar fashion to that in the case of $n=0,1$, 
the map (\ref{eqn:6.2}) implies that 
the strict transform $B^{(n+1)}$ can be locally expressed 
on $\hat{V}_0$ as 
\begin{equation}\label{eqn:6.6}
u=t,\quad v=\Phi_{n+1}(t)+\gamma_{n+1}(t)
\quad\quad\mbox{ for $t\in\R$}, 
\end{equation}
where $\Phi_{n+1}(t)=(\Phi_n(t)-\Phi_n(0))/t
\in C^{\infty}((t))$
and $\gamma_{n+1}(t)=\gamma_{n}(t)/t
\in C((t))$. 
We remark that $\gamma_{n+1}(t)$ also has a flat property.
Note that 
$E_n$ is expressed as $x=0$, $y=\tau$ for $\tau\in\C$ 
on $\hat{V}_0$ and 
$E_*$ does not appear on $V_0$.  

Next, let us consider the case 
where the center $P_n$ is not $P_n^*$.
Choose an open neighborhood $U$ of $P_n$  
such that $P_n^*$ is not contained in $U$, 
then $B^{(n+1)}$ is expressed  
in the same form as in (\ref{eqn:6.5})
on a local chart $U_0-\{P_n\}$ of $X_{n+1}$.

From the above inductive process, 
a series of blowings up gives 
a series of the strict transforms
$\{B^{(n)}\}_{n\in\Z_+}$:
\begin{equation}\label{eqn:6.7}
\xymatrix@=18pt{
\cdots \ar[r]^{\!\!\!\!\!\sigma_{n+1}}
&X_n \ar[r]^{\sigma_n} \ar@{}[d]|{\bigcup} 
&X_{n-1} \ar[r]^{\,\,\sigma_{n-1}} \ar@{}[d]|{\bigcup} 
& \cdots \ar[r]^{\sigma_2} 
& X_1 \ar[r]^{\pi_1} \ar@{}[d]|{\bigcup}
& X_0 \ar@{}[r]|{:=} \ar@{}[d]|{\bigcup}
& U\\
\cdots \ar[r]^{\!\!\!\!\!\sigma_{n+1}}
&B^{(n)} \ar[r]^{\sigma_n} 
& B^{(n-1)} \ar[r]^{\,\,\,\,\,\,\sigma_{n-1}}  
& \cdots \ar[r]^{\sigma_{2}}
& B^{(1)} \ar[r]^{\sigma_{1}}
& B^{(0)}\ar@{}[r]|{:=}
& B.
}
\end{equation}

Now, let $J_{\pi_n}$ be
the Jacobian of the composition map $\pi_n$ and let
\begin{equation}\label{eqn:6.8}
m(n):=\# \{k:P_k=P_k^*, \, k=0,\ldots,n-1\},
\end{equation}
where $\# A$ denotes the cardinal number of the set $A$. 
Then we have the following. 
\begin{lemma}
$J_{\pi_n}(x,y)=x^{m(n)}$ on the canonical coordinate.
\end{lemma}
\begin{proof}
Since the Jacobian of the map (\ref{eqn:6.2}) is $u$, 
the Jacobian of the decomposition of maps can be expressed as 
in the lemma. 
\end{proof}

To be more precise, 
the strict transform $B^{(n)}$ 
can be specifically expressed by using 
the information of  the original branch $B$.
\begin{lemma}
Let $B$ be a branch of the decisive curve 
$C_F$ defined by (\ref{eqn:6.1}), 
where $\Phi(t)\in C^{\infty}((t))$ admits
the formal Taylor series at the origin 
\begin{equation}\label{eqn:6.9}
\overline{\Phi}(t)= \sum_{j=1}^{\infty} 
c_{j}t^j,
\end{equation}
where $c_j$ are complex numbers. 
Then, the strict transform $B^{(n)}$  is locally 
expressed on some local chart on $X_n$ 
with the canonical coordinate
as 
\begin{equation}\label{eqn:6.10}
x=t,\quad y=\Phi_n(t)+\gamma_n(t)
\quad\quad\mbox{ for $t\in\R$},
\end{equation}
where $\Phi_n(t)\in C^{\infty}((t))$
admits the Taylor series 
\begin{equation}\label{eqn:6.11}
\overline{\Phi}_n(t)= \sum_{j=1}^{\infty}
c_{j+m(n)} t^j, 
\end{equation}
and $\gamma_n(t)\in C((t))$ takes the form
\begin{equation}\label{eqn:6.12}
\gamma_n(t)=\frac{\gamma_0(t)}{t^{m(n)}}.
\end{equation}
Here $m(n)$ is as in (\ref{eqn:6.8}), 
$c_j$  are the same as in (\ref{eqn:6.9})
and $\gamma_n(t)$ has a flat property.
\end{lemma}

\begin{proof}
The equation (\ref{eqn:6.10})
 is shown by induction on $n$: 
the case of $n=0$ is obvious.
Let us assume that 
$B^{(n)}$ is locally expressed by using the 
Taylor series (\ref{eqn:6.11}).
If $P_n=P_n^*$, 
then $m(n+1)=m(n)+1$ and 
the strict transform $B^{(n+1)}$ 
can be expressed as
$\Phi_{n+1}(t)=(\Phi_n(t)-\Phi_n(0))/t$
plus a flat term, 
which implies 
$$
\overline{\Phi}_{n+1}(t)=
\sum_{j=1}^{\infty}
c_{j+m(n)+1}t^j.
$$
On the other hand, if $P_n\neq P_n^*$, then
we have $m(n+1)=m(n)$ and 
$\overline{\Phi}_{n+1}(t)=\overline{\Phi}_{n}(t)$.
As a result, we see that (\ref{eqn:6.11}) holds in the case of ($n+1$). 

The equation (\ref{eqn:6.10}) can be similarly shown and 
the flat property of $\gamma_n(t)$ is obvious. 
\end{proof}



Next, 
from the inductive process of blowings up of a branch 
which was explained in this section, 
we can understand  geometrical situations of 
the strict transform of every branch of $C_F$.
Let 
${\mathcal B}=\{B_1,\ldots, B_l\}$
be the set of the branches of $C_F$ and 
let 
${\mathcal B}^{(n)}=\{B_1^{(n)},\ldots, B_l^{(n)}\}$ 
be the set of their $n$-th strict transforms. 
We denote ${\mathcal B}={\mathcal B}^{(0)}$
and $B_j=B_j^{(0)}$ for $j=1,\ldots,l$. 
Recall that ${\mathcal E}^{(n)}$ is the set of 
the exceptional curves 
of $\pi_n$ in $X_n$, 
which is defined in Section~4.4.

\begin{proposition}
\begin{enumerate}
\item 
Every strict transform $B_j^{(n)}\in {\mathcal B}^{(n)}$ 
meets one exceptional 
curve of $\pi_n$ at a single point.
We denote 
${\mathcal P}^{(n)}=
\{B_j^{(n)}\cap E:B_j^{(n)}\in{\mathcal B}^{(n)}, 
E\in{\mathcal E}^{(n)}\}.$
\item
Conversely, for a given $P\in{\mathcal P}^{(n)}$, 
$I(P)\subset \{1,\ldots,l\}$ and 
$E(P)\in{\mathcal E}^{(n)}$ are defined as follows: 
$B_j^{(n)}\cap E=\{P\}$ if and only if $j\in I(P)$ 
and $E=E(P)$.
For any $P\in{\mathcal P}^{(n)}$, 
there exists a local chart
$(U,\varphi)$ with canonical coordinate such that 
\begin{enumerate}
\item $U$ contains $P$;
\item $E(P)$ is expressed as 
$x=0$, $y=\tau$ for $\tau\in\C$ on $\varphi(U)$;
\item 
Each $B_j^{(n)}$ with $j\in I(P)$ 
is locally expressed on $\varphi(U)$ as in the form: 
\begin{equation}\label{eqn:6.13}
x=t, \quad y=\Phi_j(t)+\gamma_j(t)
\quad\quad\mbox{ for $t\in\R$},
\end{equation}
where $\Phi_j(t)\in C^{\infty}((t))$ and 
$\gamma_j(t)\in C((t))$ has
a flat property.
\end{enumerate}
\end{enumerate}
\end{proposition}

\begin{proof}
The assertion of this proposition has been essentially shown 
by induction in this section. 
We remark that when the blowing up 
$\sigma_{n+1}:X_{n+1}\to X_n$ is constructed, 
on open neighborhood $U$ must be chosen so that
$U$ contains at most one point in ${\mathcal P}^{(n)}$. 
\end{proof}

\subsection{Exponent of contact of the strict transforms of two branches}

Let us observe how the geometrical relationship between
the strict transforms of two branches of $C_F$ 
are changed by a series of blowings up in (\ref{eqn:4.7}).


Let us introduce some quantity 
which measures the strength 
of the contact of two strict transforms. 
Let $n$ be a nonnegative integer and
let
$B^{(n)}$, $\tilde{B}^{(n)}$ be branches 
belonging to ${\mathcal B}^{(n)}$. 
When 
$B^{(n)}\cap\tilde{B}^{(n)}\cap \pi_n^{-1}(O)\neq\emptyset$,
it follows from Proposition~6.3 that 
there exists a local chart $U$ of $X_n$ such that
$B^{(n)}$, $\tilde{B}^{(n)}$ are locally expressed on the canonical coordinate
as 
$x=t$,  $y=\Phi(t)+\gamma(t)$ and 
$x=t$,  $y=\tilde{\Phi}(t)+\tilde{\gamma}(t)$ where 
$\Phi(t), \tilde{\Phi}(t)\in C^{\infty}((t))$ 
admit the formal Taylor series: 
\begin{equation}\label{eqn:6.14}
\overline{\Phi}(t)=\sum_{j=0}^{\infty} c_j t^j,
\quad 
\overline{\tilde{\Phi}}(t)=\sum_{j=0}^{\infty}
\tilde{c}_j t^j, 
\end{equation}
with $c_j,\tilde{c}_j$ are complex numbers, 
and $\gamma(t), \tilde{\gamma}(t)\in C((t))$ have a flat property.  
Note that $c_0=\tilde{c}_0$.
\begin{definition}
For $B^{(n)},\tilde{B}^{(n)}\in{\mathcal B}^{(n)}$,  
we define
$
{\mathcal O}(B^{(n)},\tilde{B}^{(n)})
$
as follows. 
\begin{enumerate}
\item
If
$B^{(n)}\cap\tilde{B}^{(n)}\cap \pi_n^{-1}(O)=\emptyset$,
then set ${\mathcal O}(B^{(n)},\tilde{B}^{(n)})=0$;
\item
If 
$B^{(n)}\cap\tilde{B}^{(n)}\cap \pi_n^{-1}(O)\neq\emptyset$
and $\overline{\Phi}(t) =\overline{\tilde{\Phi}}(t)$, 
then set 
${\mathcal O}(B^{(n)},\tilde{B}^{(n)})=\infty$;
\item
If
$B^{(n)}\cap\tilde{B}^{(n)}\cap \pi_n^{-1}(O)\neq\emptyset$ and 
$\overline{\Phi}(t)\neq \overline{\tilde{\Phi}}(t)$, 
then set 
${\mathcal O}(B^{(n)},\tilde{B}^{(n)})=
\min\{j\in\N: c_j \neq \tilde{c}_j\},$
where $c_j, \tilde{c}_j$ are as in (\ref{eqn:6.14}).
\end{enumerate}
We call ${\mathcal O}(B^{(n)},\tilde{B}^{(n)})$ 
the {\it exponent of contact} of $B^{(n)},\tilde{B}^{(n)}$.
\end{definition}

The exponent of contact of
$B^{(n)},\tilde{B}^{(n)}$
stands for the strength of contact of 
$B^{(n)}$ and $\tilde{B}^{(n)}$ at some point on $\pi_n^{-1}(O)$.
In particular,  
${\mathcal O}(B^{(n)},\tilde{B}^{(n)})=0$ 
means that $B^{(n)}$ is separated from
$\tilde{B}^{(n)}$ near $\pi_n^{-1}(O)$; while 
${\mathcal O}(B^{(n)},\tilde{B}^{(n)})=\infty$ 
means that $B^{(n)}$ is infinitely tangent to
$\tilde{B}^{(n)}$ at some point in ${\mathcal P}^{(n)}$.
Since the following two lemmas concerning the above 
quantity can be easily shown from its definition, 
their proofs will be left to the readers. 

\begin{lemma}
For $B^{(n)}$, $\tilde{B}^{(n)}$, 
$\tilde{\tilde{B}}^{(n)}\in{\mathcal B}^{(n)}$
with $n\in\Z_+$, 
then the following holds:
\begin{enumerate}
\item
${\mathcal O}(B^{(n)},\tilde{B}^{(n)})=
{\mathcal O}(\tilde{B}^{(n)},B^{(n)})$; 
\item
${\mathcal O}(B^{(n)}, B^{(n)})=\infty$; 
\item
If
${\mathcal O}(B^{(n)}, \tilde{B}^{(n)})=\infty$,
then
${\mathcal O}(B^{(n)}, \tilde{\tilde{B}}^{(n)}) =
{\mathcal O}(\tilde{B}^{(n)}, \tilde{\tilde{B}}^{(n)})$;
\item
${\mathcal O}(B^{(n)}, \tilde{\tilde{B}}^{(n)})
\geq 
\min\{
{\mathcal O}(B^{(n)}, \tilde{B}^{(n)}),
{\mathcal O}(\tilde{B}^{(n)}, \tilde{\tilde{B}}^{(n)}) 
\}.$
\end{enumerate}
\end{lemma}

\begin{lemma}
For $B,\tilde{B}\in{\mathcal B}$, 
the following three conditions are equivalent:
\begin{enumerate}
\item ${\mathcal O}(B,\tilde{B})=\infty$;
\item ${\mathcal O}(B^{(n)},\tilde{B}^{(n)})=\infty$
for some $n\in\Z_+$;
\item ${\mathcal O}(B^{(n)},\tilde{B}^{(n)})=\infty$
for any $n\in\Z_+$.
\end{enumerate}
\end{lemma}


From Lemma~6.5 (i), (ii), (iii),
an equivalence relationship ``$\sim$'' can be 
introduced in ${\mathcal B}^{(n)}$ as follows:
\begin{equation}\label{eqn:6.15}
B^{(n)}\sim \tilde{B}^{(n)}
\overset{\text{def}}{\Longleftrightarrow}
{\mathcal O}(B^{(n)}, \tilde{B}^{(n)})=\infty.
\end{equation}
It follows from Lemma~6.6 that
$B\sim\tilde{B}$
if and only if 
$B^{(n)}\sim\tilde{B}^{(n)}$
for any (or some) $n$.

The following lemma is important. 
Notice that if
${\mathcal O}(B^{(n)}, \tilde{B}^{(n)})$ is a positive integer, 
then 
$B^{(n)}$ and $\tilde{B}^{(n)}$ intersect 
on $\pi_n^{-1}(O)$ and $B^{(n)}\not\sim\tilde{B}^{(n)}$.
\begin{lemma}
Let 
$B, \tilde{B}\in{\mathcal B}$ 
satisfy that ${\mathcal O}(B^{(n)}, \tilde{B}^{(n)})$ is a positive 
integer for some $n\in\Z_+$.
If $\sigma_{n+1}:X_{n+1}\to X_n$ is a blowing up 
with center the point $P\in{\mathcal P}^{(n)}$, at which 
$B^{(n)}$ and $\tilde{B}^{(n)}$ intersect, then
we have
\begin{equation}\label{eqn:6.16}
{\mathcal O}(B^{(n+1)},\tilde{B}^{(n+1)})=
{\mathcal O}(B^{(n)}, \tilde{B}^{(n)})-1.
\end{equation}
\end{lemma}

\begin{proof}
Let (\ref{eqn:6.14}) be the formal Taylor series 
characterizing the branches $B^{(n)}, \tilde{B}^{(n)}$, 
respectively. 
Then, it follows from Lemma~6.2 that
$\sum_{j=0}^{\infty} c_{j+1} t^j$,  
$\sum_{j=0}^{\infty} \tilde{c}_{j+1} t^j$ are 
the Taylor series characterizing 
$B^{(n+1)}$, $\tilde{B}^{(n+1)}$, respectively. 
Therefore, we see that
\begin{equation*}
\begin{split}
&{\mathcal O}(B^{(n+1)},\tilde{B}^{(n+1)})=
\min\{j:c_{j+1}\neq \tilde{c}_{j+1}\} \\
&\quad\quad\quad\quad\quad
=\min\{j :c_{j}\neq \tilde{c}_{j}\}-1=
{\mathcal O}(B^{(n)}, \tilde{B}^{(n)})-1.
\end{split}
\end{equation*}
\end{proof}

As a corollary of Lemma~6.7, we obtain the following.

\begin{proposition}
There exists a finite series of blowings up 
\begin{equation}\label{eqn:6.17}
\begin{CD}
     X_{N} @>{\sigma_{N}}>> X_{N-1}  
@>{\sigma_{N-1}}>> \cdots
@>{\sigma_{2}}>> X_1
@>{\sigma_{1}}>> X_0:=U
  \end{CD}
\end{equation}
such that 
${\mathcal O}(B^{(N)},\tilde{B}^{(N)})=0$
for every pair of branches 
$B, \tilde{B}\in{\mathcal B}$ with
$B\not\sim\tilde{B}$. 
\end{proposition}

\begin{proof}
Let $M(n)$ be a nonnegative integer defined by 
$M(n)=\sum {\mathcal O}(B^{(n)},\tilde{B}^{(n)})$,
where the summation is taken over all the
pairs of branches 
$B,\tilde{B}\in{\mathcal B}$ 
with
$B\not\sim\tilde{B}$.
It is easy to see that $M(0)<\infty$ and that 
$M(n+1)\leq M(n)$ for any $n\in\Z_+$.
It suffices to show the existence of a positive integer
$N$ such that $M(N)=0$. 
If $M(n)>0$, then there exist $B, \tilde{B}\in{\mathcal B}$
with $B\not\sim\tilde{B}$
and a point $P_*\in\pi^{-1}(O)$ such that
$B^{(n)}$ and $\tilde{B}^{(n)}$ intersect at $P_*$.
From Lemma~6.7,
the blowing up of $X_n$ with center $P_*$
gives 
${\mathcal O}(B^{(n+1)},\tilde{B}^{(n+1)})=
{\mathcal O}(B^{(n)},\tilde{B}^{(n)})-1$, 
which implies $M(n+1)\leq M(n)-1$. 
By the inductive process, 
the existence of a positive integer $N$ such that
$M(N)=0$ can be shown. 
\end{proof}

Recall that all the branches of $C_F$
can be parametrized by using 
the Rychkov factorization formula (\ref{eqn:5.5}).
Indeed, the set of the branches of $C_F$ is expressed as 
${\mathcal B}=\{B_{jk}: j=1,\ldots, r, \,\,
k=1,\ldots,m_j\}$,
where each
$B_{jk}$ is locally parametrized as   
$x=t$, $y=\Phi_j(t)+\gamma_{jk}(t)$ 
for $t\in\R$, where 
$\Phi_j(t), \gamma_{jk}(t)$ are as in Proposition~5.1. 
It is easy to determine the equivalence of the elements 
in ${\mathcal B}$ as follows. 
\begin{lemma}
The following three conditions are equivalent.
\begin{enumerate}
\item $B_{jk}\sim B_{j' k'}$;
\item $B_{jk}^{(n)} \sim B_{j' k'}^{(n)}$ 
for any $n\in\Z_+$;
\item $j=j'$.
\end{enumerate}
\end{lemma}

\begin{proof}
The above equivalences can be seen from Lemma~6.6 and 
the definition (\ref{eqn:6.15}) 
of the equivalence relationship ``$\sim$". 
\end{proof}

On the complex manifold $X_N$ constructed in Proposition~6.8,
local geometrical situations 
of the $N$-th strict transforms of the branches can be clearly understood.  
Recall that ${\mathcal P}^{(N)}$ is as in Proposition~6.3.

\begin{proposition}
\begin{enumerate}
\item
For each $j$,  all the $B_{jk}^{(N)}\cap \pi_N^{-1}(O)$
for $k=1,\ldots, m_j$ are the same point on $X_N$, 
which will be denoted by $P_j^{(N)}$,  
i.e.,  ${\mathcal P}^{(N)}=\{P_j^{(N)}:j=1,\ldots,r\}$
(${\mathcal P}^{(N)}$ is as in Proposition~6.3).  
Moreover, if $j\neq j'$, then $P_j^{(N)}\neq P_{j'}^{(N)}$. 
\item
For any $P_j^{(N)}\in{\mathcal P}^{(N)}$, 
there exists a local chart $(U,\varphi)$ with 
the canonical coordinate such that
\begin{enumerate}
\item $U$ contains $P_j^{(N)}$;
\item The exceptional curve containing $P_j^{(N)}$ 
is locally expressed as $x=0$, $y=\tau$ for 
$\tau\in\C$ on $\varphi(U)$;
\item For $k=1,\ldots,m_j$, 
each $B_{jk}^{(N)}$ is locally 
expressed on $\varphi(U)$ as 
\begin{equation}\label{eqn:6.18}
x=t,\,\, y=\Phi_j^{(N)}(t)+\gamma_{jk}^{(N)}(t)
\quad\quad\mbox{for $t\in\R$},
\end{equation}
where $\Phi_j^{(N)}(t)\in C^{\infty}((t))$ 
and $\gamma_{jk}^{(N)}(t)\in C((t))$ has a flat property. 
\end{enumerate}
\end{enumerate}
\end{proposition}

\begin{proof}
Applying Proposition~6.8 and Lemma~6.9 to Proposition~6.3, 
we obtain this proposition.  
\end{proof}

\begin{remark}
From Lemma~6.2, 
the functions $\Phi_j^{(N)}(t)$ and 
$\gamma_{jk}^{(N)}(t)$
in (\ref{eqn:6.18}) are specifically expressed 
by using the information of $\Phi_j(t)$ and $\gamma_{jk}(t)$. 
In particular, 
there exists a positive integer $L$ such that
$\gamma_{jk}^{(N)}(t)=\gamma_{jk}(t)/t^L$ for any $k$.
Therefore, for $\alpha\in\N$, $j=1,\ldots,r$,  we have
\begin{equation}\label{eqn:6.19}
{\mathcal E}_{j\alpha}^{(N)}(t):=
\sum_{k=1}^{m_j} [\gamma_{jk}^{(N)}(t)]^{\alpha}
=\frac{1}{t^{L \alpha}} \sum_{k=1}^{m_j} 
[\gamma_{jk}(t)]^{\alpha}.
\end{equation}
Proposition 5.1 (iv) implies that  
${\mathcal E}_{j\alpha}^{(N)}(t)$ belong to $C^{\infty}((t))$
for $j=1,\ldots, r$. 
Furthermore, if $\overline{\Phi}_j(t)\in\R[[t]]$, then
we can see that 
${\mathcal E}_{j\alpha}^{(N)}(t)$ belong to $\R C^{\infty}((t))$
for $j=1,\ldots, r$
from Proposition 5.1 (v).
These properties will be used later. 
\end{remark}


\subsection{Removement of non-real branches}

Let $B$ be a non-real branch locally parametrized 
as in (\ref{eqn:6.1}),
where $\Phi_0(t)$ admits the formal Taylor series 
$$
\overline{\Phi}_0(t)=
\sum_{j=1}^{\infty} c_j t^j.
$$ 
Since $B$ is a non-real branch, 
the following positive integer can be decided:
\begin{equation}
m:=\min\{j\in\N:
c_j\in\C-\R\}.
\end{equation}
Let us consider a series of blowings up 
satisfying that the center of each blowing up $\sigma_k$ 
is $(0,\Phi_k(0))=(0,c_k)$ on the canonical coordinates.
When $m\geq 1$, since each $\Phi_k(0)=c_k$ is real for 
$k=0,1,\ldots,m-1$, we can construct
a finite series of blowings up with real centers 
and their composition map $\pi_m$ in (\ref{eqn:4.8}).
Simultaneously, 
a finite series of real blowings up (\ref{eqn:4.9}) and 
their composition map
$\hat{\pi}_m$ in (\ref{eqn:4.10}) can also be constructed. 
Let $\pi_0$, $\hat{\pi}_0$ denote the identity maps. 

\begin{lemma}
There exists 
 an open neighborhood $\hat{U}$ of the origin
in $\R^2$ 
such that 
$B^{(m)}\cap {\hat{\pi}_m}^{-1}(\hat{U})=\emptyset$.
Here, ${\hat{\pi}_m}^{-1}(\hat{U})$ 
$(\subset Y_m)$ is regarded as a subset in $X_m$.
\end{lemma}

\begin{proof}
It follows from Lemma~6.2 that 
the strict transform $B^{(m)}$ is locally expressed
as (\ref{eqn:6.10}) with (\ref{eqn:6.11}), 
where $n$ in (\ref{eqn:6.10}) and 
$m(n)$ in (\ref{eqn:6.11}) are replaced by $m$. 
From these equations, we have $\Phi_m(t)=c_m+O(t)$.
Since $c_m$ is a non-real complex number, 
the continuity of $\Phi_m(t)$ implies the existence of
$\delta>0$
such that $\Phi_m(t)\not\in\R$ for 
$t\in (-\delta,\delta)$. 
Let $\hat{U}=(-\delta,\delta)\times
(-\delta,\delta)$, then
we can see that 
$B^{(m)}\cap {\hat{\pi}_m}^{-1}(\hat{U})=\emptyset$.
\end{proof}

As a corollary of the above lemma, 
the following proposition can be easily shown.

\begin{proposition}
There exist an open neighborhood $U$ of the origin in $\C^2$ 
and a finite series of blowings up with real centers:
\begin{equation}\label{eqn:6.21}
\begin{CD}
     X_{N} @>{\sigma_{N}}>> X_{N-1}  
@>{\sigma_{N-1}}>> \cdots
@>{\sigma_{2}}>> X_1
@>{\sigma_{1}}>> X_0:=U 
  \end{CD}
\end{equation}
such that 
$B^{(N)}\cap {\hat{\pi}_N}^{-1}(\hat{U})=\emptyset$
for any non-real branches $B$ of $C_F$, 
where $\hat{U}$ is the restriction of $U$ to $\R^2$. 
\end{proposition} 

\begin{proof}
Let $B_1,\ldots,B_L$ be the non-real branches 
of $C_F$. 
It suffices to show the following: 
For any $n\in\{1,\ldots, L\}$, 
there exist an $m_n\in\N$, 
a finite series of blowings up (\ref{eqn:6.17}) with real centers 
with $N=m_n$ and an open neighborhood $\hat{U}_n$ in 
$\R^2$ such that 
\begin{equation}\label{eqn:6.22}
B_j^{(m_n)}\cap \hat{\pi}_{m_n}^{-1}(\hat{U}_n)
=\emptyset
\quad\quad \mbox{for $j=1,\ldots,n$.}
\end{equation}

Let us show the above by induction: the case of $n=1$ 
has been shown in Lemma~6.12. 
We assume that (\ref{eqn:6.22}) holds.
If $B_{n+1}^{(m_n)}\cap\hat{\pi}_{m_n}^{-1}(O)=\emptyset$,
then we can choose an open neighborhood 
$\hat{U}_{n+1}(\subset \hat{U}_{n})$
of the origin 
such that $B_{n+1}^{(m_n)}\cap\hat{\pi}_{m_n}^{-1}(U_{n+1})
=\emptyset$. 
In this case, it suffices to set $m_{n+1}=m_n$.
If $B_{n+1}^{(m_n)}\cap\hat{\pi}_{m_n}^{-1}(O)\neq\emptyset$,
then the strict transform $B_{n+1}^{(m_n)}$ can be
locally expressed on the canonical coordinate by using 
a formal series of the form 
$\sum_{j=0}^{\infty} d_j t^j \in\C[[t]]$.
Here, we can define a positive integer 
$M:=\min\{j:d_j\in\C-\R\}$ from the condition
$B_{n+1}^{(m_n)}\cap\hat{\pi}_{m_n}^{-1}(O)\neq\emptyset$. 
Then, by applying the same blowing up process
as that in Lemma~6.12, we can see the existence
of the series of blowings up with real centers
such that
(\ref{eqn:6.22}) holds in the case of $(n+1)$.
Here, we have $m_{n+1}=m_{n}+M$.  
\end{proof}


\subsection{In the case of $\mu_0(F)=0$}

Let us consider the case where there is no real branches
of the decisive curves $C_F$.
From Proposition~6.13, 
we can obtain the following proposition, 
which implies that 
$F$ can be locally expressed in
the normal crossings form by means of
an appropriate series of blowings up with real centers. 

\begin{proposition}
If $\mu_0(F)=0$ 
(i.e., every branch of the curve $C_F$ is non-real), 
then there exist 
an open neighborhood $\hat{U}$ of the origin in $\R^2$,  
a two-dimensional $C^{\omega}$ real manifold ${Y}$ and 
a proper $C^{\omega}$ map $\hat{\pi}$ from ${Y}$ 
to $\hat{U}$ such that 
\begin{enumerate}
\item 
$\hat{\pi}$ is a local isomorphism from 
${Y} - \hat{\pi}^{-1}(O)$ to 
$\hat{U} - \{O\}$. 
\item For each 
$Q\in \hat{\pi}^{-1}(O)$, 
there are 
local $C^{\omega}$ coordinates $(x,y)$ centered at $Q$ so that
the following hold:
\begin{equation}
(F\circ \hat{\pi}) (x,y)=
u_Q(x,y) x^{A_Q} y^{B_Q} 
\quad \mbox{and} \quad 
J_{\hat{\pi}}(x,y)=x^{C_Q},
\end{equation}
where $u_Q\in \R C^{\infty}((x,y))$ satisfies 
$u_Q(0,0)\neq 0$ and 
$A_Q, B_Q, C_Q$ are nonnegative integers. 
\end{enumerate}
\end{proposition}

\begin{proof}
From Proposition~6.13, 
there exists a map 
$\hat{\pi}:{Y}\to \hat{U}$ constructed 
from the composition of a series of 
blowings up, 
which satisfies that  
 $\hat{\pi}^{-1}(C_F)$ 
consists of only real exceptional curves in 
$\hat{\pi}^{-1}(\hat{U})$. 
It is easy to see that the exceptional curves is locally given by 
the zero locus
of the function of the normal crossings on local charts. 
Since $\hat{\pi}$ is a composition of the maps 
in (\ref{eqn:4.10}) with (\ref{eqn:6.21}), 
the Jacobian of $\hat{\pi}$ can be written as in 
Lemma~6.1.
\end{proof}

\subsection{In the case of $\mu_0(F)\geq 1$}

Let us consider the case where at least one real branch exists 
in the decisive curve $C_F$. 
Recall ${\mathcal R}(F):=
\{j:\overline{\Phi}_j(t)\in\R[[t]]\}$,
where $\overline{\Phi}_j(t)$ are as in (\ref{eqn:3.1}). 
Note that ${\mathcal R}(F)\neq\emptyset$ 
if $\mu_0(F)\geq 1$.

\begin{proposition}
If $\mu_0(F)\geq 1$, then 
there are 
an open neighborhood $\hat{U}$ of the origin in $\R^2$,
a two-dimensional $C^{\omega}$ real manifold ${Y}$ and 
a proper $C^{\omega}$ map $\hat{\pi}$ from ${Y}$ 
to $\hat{U}$ such that 
\begin{enumerate}
\item 
$\hat{\pi}$ is a local isomorphism from 
${Y} - \hat{\pi}^{-1}(O)$ to 
$\hat{U} - \{O\}$. 
\item
For each $j\in{\mathcal R}(F)$, 
there exist a point ${P}_j \in \hat{\pi}^{-1}(O)$ and
a local $C^{\omega}$ coordinate $(x,y)$ 
centered at ${P}_j$ so that
the following (a), (b) hold:
\begin{enumerate}
\item
$(F\circ \hat{\pi})(x,y)$ is locally expressed as 
\begin{equation}\label{eqn:6.24}
(F\circ \hat{\pi}) (x,y)=
u_j(x,y)
x^{a_j} \prod_{k=1}^{m_{j}} 
(y-\tilde{\Phi}_{j}(x)-\tilde{\gamma}_{jk}(x)),
\end{equation}
where
${u}_j(x,y)\in \R C^{\infty}((x,y))$ satisfies 
${u}_j(0,0)\neq 0$, 
$a_j$ is a nonnegative integers, 
$m_j$ is as in (\ref{eqn:3.1}), 
$\tilde{\Phi}_{j}(x)$ belongs to $\R C^{\infty}((x))$ and
$\tilde{\gamma}_{jk}(x)\in C ((x))$ ($k=1,\ldots, m_j$)
satisfy that 
$\tilde{\gamma}_{jk}(x)=O(x^l)$ as $x\to 0$
for any $l\in\N$.
\item
The Jacobian of $\hat{\pi}$ is locally expressed as 
\begin{equation}\label{eqn:6.25}
J_{\hat{\pi}}(x,y)=x^{M_j},
\end{equation}
where $M_j$ is a nonnegative integer. 
\end{enumerate}
\item For each 
$Q\in \hat{\pi}^{-1}(O)-
\{P_j:j\in{\mathcal R}(F)\}$,
there exists a 
local coordinate $(x,y)$ centered at $Q$ so that
the following locally hold:
\begin{equation}\label{eqn:6.26}
(F\circ \hat{\pi}) (x,y)=u_Q(x,y) x^{A_Q} y^{B_Q}, 
\quad \mbox{and} \quad 
J_{\hat{\pi}}(x,y)= x^{C_Q},
\end{equation}
where ${u}_Q(x,y)\in \R C^{\infty}((x,y))$ satisfies 
${u}_Q(0,0)\neq 0$ and 
$A_Q,B_Q,C_Q$ are nonnegative integers. 
\end{enumerate}

\end{proposition}


\begin{proof}
For a given $F\in\R C^{\infty}((x,y))$ with $\mu_0(F)\geq 1$, 
we can construct a series of blowings up with real centers 
and an open neighborhood of $U$ of the origin in $\C^2$ 
as in Proposition~6.13. 
We denote the composition map of this series of blowings up
by $\pi_A:X_A\to U$, where 
$X_A:=X_N$ which is as in Proposition~6.13. 
Moreover, we denote by $\hat{\pi}_A:Y_A\to \hat{U}$
the $C^{\omega}$ map which are simultaneously constructed 
in the real version. 
Here $\hat{U}$ is the restriction of $U$ to $\R^2$ and
$Y_A=\hat{X}_A$. 
Since all the strict transforms of non-real branches in $X_A$
are away from $\hat{\pi}_A^{-1}(\hat{U})$, 
it suffices to deal with the real branches $B_{jk}$ of $C_F$.

Next, using the series of blowings up in Proposition~6.8, 
we can construct a series of blowings up 
with real centers $X_{N'} \to X_{N'-1} \to \cdots \to X_A$
such that ${\mathcal O}(B_{jk}^{(N')}, B_{j' k'}^{(N')})=0$
if $j\neq j'$ and 
${\mathcal O}(B_{jk}^{(N')}, B_{j' k'}^{(N')})=\infty$
if $j= j'$, where $B_{jk}, B_{j' k'}$ are real branches 
of $C_F$. 
The composition map of this series of blowings up
as $\pi_B:X_B\to X_A$, where $X_B=X_{N'}$.
We can simultaneously construct
the $C^{\omega}$ map $\hat{\pi}_B:Y_B\to Y_A$ 
in the real version.
Here $Y_B=\hat{X}_B$.

Now, we respectively denote 
the maps 
$\pi:=\pi_A\circ\pi_B$ and 
$\hat{\pi}:=\hat{\pi}_A \circ\hat{\pi}_B$. 
Let us show that $\hat{\pi}:Y_B\to \hat{U}$ is a
desired map in the proposition. 

(i)\quad 
Since $\hat{\pi}$ is a composition of finite numbers
of real blowing maps, the property (i) can be easily seen.

(ii)\quad 
Since each branches of the zero variety of $F\circ\hat{\pi}$ 
can be parameterized 
as in (\ref{eqn:6.18}) in Proposition~6.10,  
we obtain the equation (\ref{eqn:6.24}). 
Moreover, the Jacobian of $\hat{\pi}$ can be expressed 
as (\ref{eqn:6.25}) from Lemma~6.1.


(iii)\quad 
In a similar fashion to the proof of Proposition~6.14, 
(iii) can be shown. 
\end{proof}

\section{Proof of Theorem~3.2}

In comparison to Proposition~6.15, 
it essentially suffices to show (ii) in Theorem~3.2. 
Indeed, we choose
the local coordinate $(\hat{x},\hat{y})$ around 
each $P_j$
defined by
\begin{equation}\label{eqn:7.1}
(\hat{x},\hat{y})
=(x,y-\tilde{\Phi}_j(x)),
\end{equation}
where $(x,y)$ is a local coordinate around $P_j$ as in Proposition~6.15. 
Then we can show the claim in (ii-a) in the theorem as follows. 
We remark that since the above transform 
may not be real analytic, 
the manifold $Y$ may have only 
the $C^{\infty}$ structure. 

First, let us show the equation in (\ref{eqn:3.3}).
Recall that $\sum_{k=1}^{m_j} 
[\tilde{\gamma}_{jk}(t)]^{\alpha}$ 
belong to $\R C((t))$ for $\alpha\in\N$ 
in Remark~6.11.
From this property, 
all the elementary symmetric polynomials 
in the variables 
$\tilde{\gamma}_{j1}(t),\ldots,
\tilde{\gamma}_{jm_j}(t)$ belong to 
$\R C^{\infty}((t))$, which 
implies that 
$\varepsilon_{jk}(t)$ in (\ref{eqn:3.3}) 
belong to $\R C^{\infty}((t))$ 
for $k=1,\ldots,m_j$.

Next, since the transformation (\ref{eqn:7.1}) does not affect 
the properties of the Jacobian, we can see (\ref{eqn:3.4}) 
in Theorem~3.2. 

\section{Properties of $\mu_0(f)$}

In this section, we investigate more detailed properties of 
the quantity $\mu_0(f)$ introduced in Section~2.4.

\subsection{The geometric meaning}
For simplicity, 
we only consider the case of the function $F\in\R C^{\infty}((x,y))$ 
satisfying that its Taylor series admits the factorization (\ref{eqn:3.1}).
(The argument below can be naturally extended
to the general case.) 

First, let us consider the case where 
$F(x,y)$ is real analytic.
In this case,  
the branches $B_{jk}$ are the same object in $\C^2$
for $k=1,\ldots,m_j$ if $j$ is fixed.
Thus we denote $B_j:=B_{jk}$ for 
$k=1,\ldots,m_j$. 
By using the language of the divisors, 
the decisive curve $C_F$ may be written as
the sum of effective divisors, i.e., 
\begin{equation}\label{eqn:8.1}
C_F=\sum_{j=1}^r m_j B_j
=\sum_{j\in{\mathcal R}(F)} m_j B_j+
\sum_{j\not\in{\mathcal R}(F)} m_j B_j.
\end{equation}
In the $C^{\infty}$ case, 
although slight gaps may be made by flat functions, 
it can be interpreted that 
the equations (\ref{eqn:8.1}) 
hold in the level of formal power series
or in the equivalence class introduced in (\ref{eqn:6.15}).
From this point of view,  it might be said that 
the quantity $\mu_0(F)$ indicates 
the maximum of the multiplicities 
of the {\it formal} real branches of $C_F$.

\subsection{The algebraic meaning}

Let $f(x,y)\in\R C^{\infty}((x,y))$ admit 
the Taylor series $\overline{f}(x,y)\in\R[[x,y]]$. 
From the algebraic property of the ring 
of formal power series $\C[[x,y]]$, 
let us explore the property of $\mu_0(f)$. 
The zero element of the ring of formal power series 
is denoted by $\overline{0}$.

For $\overline{P}(x,y)\in\C[[x,y]]$, 
${\mathcal S}_{\R}(\overline{P})$ denotes
the set of the pairs $(\overline{\phi}(t), \overline{\psi}(t))$ satisfying 
\begin{itemize}
\item $\overline{\phi}(t), \overline{\psi}(t)\in\R[[t]]$;
\item At least one of $\overline{\phi}(t), \overline{\psi}(t)$ 
is different from $\overline{0}$;
\item
$\overline{P}(\overline{\phi}(t), \overline{\psi}(t))=\overline{0}$.
\end{itemize}
Two pairs 
$(\overline{\phi}(t), \overline{\psi}(t))$ and 
$(\overline{\phi}'(t), \overline{\psi}'(t))$ in ${\mathcal S}_{\R}(P)$ 
are equivalent 
if there exists a unit $\overline{\rho}(t)\in\R[[t]]$
such that 
$(\overline{\phi}'(t), \overline{\psi}'(t))=
(\overline{\phi}(\overline{\rho}(t)), \overline{\psi}(\overline{\rho}(t)))$. 
An equivalence class in the equivalence relation defined above
in ${\mathcal S}_{\R}(P)$ is called a {\it real root} of $\overline{P}$.

\begin{lemma}
If $\overline{P}(x,y)\in \C[[x,y]]$ is irreducible, then
there exists at most one real root of $\overline{P}$.
Furthermore, 
if $\overline{P}$ has a real root, then
one of the following two conditions holds:
\begin{enumerate}
\item There exist $n\in\N$ and $\varphi(t)\in\R[[t]]$ 
such that $\overline{P}(t^n,y)$ is divisible by $(y-\overline{\varphi}(t))$;
\item $\overline{P}(x,y)$ is divisible by $x$.
\end{enumerate}
\end{lemma}


\begin{proof}
If there exists a positive integer $n$ such that 
$\overline{P}(0,y)$ is of order $n$
(i.e., $P(0,y)=cy^n+\cdots$, with $c\neq 0$), 
then the Puiseux theorem implies 
\begin{equation}\label{eqn:8.2}
\overline{P}(t^n,y)=\overline{U}(t^n,y)
\prod_{k=0}^{n-1} (y-\overline{\varphi}(\omega^k t)),
\end{equation}
where $\overline{U}(x,y)\in\C[[x,y]]$ is a unit, 
$\overline{\varphi}(t)$ belongs to $\C[[t]]$ and 
$\omega=e^{2\pi i/n}$.
From the above factorization, 
there exists at most one $\overline{\varphi}(t)\in\R[[t]]$ 
such that $\overline{P}(t^n,\overline{\varphi}(t))=0$, 
which implies $\overline{P}$ has at most one real root. 
On the other hand, if $\overline{P}(0,y)=0$, then 
$\overline{P}(x,y)$ can be represented as 
$\overline{P}(x,y)=\overline{U}(x,y)x$, 
where $\overline{U}(x,y)$ is a unit.
In this case, $(0,t)$ expresses only one real roof of 
$\overline{P}$. 
\end{proof}

Since $\C[[x,y]]$ is a unique factorization domain, 
$\overline{f}(x,y)$ can be expressed as 
\begin{equation}\label{eqn:8.3}
\overline{f}(x,y)=
\overline{P}_1^{m_1}(x,y)\cdots\overline{P}_l^{m_l}(x,y),
\end{equation}
where $m_j$ are positive integers 
and the factors $\overline{P}_j(x,y)\in\C[[x,y]]$ 
are irreducible and distinct
(we cannot write $\overline{P}_j=\overline{V}\cdot\overline{P}_k$ 
for any unit $\overline{V}$ if $j\neq k$).

Now let $\tilde{\mathcal R}(\overline{f}):=
\{j:\overline{P}_j \mbox{ has
a real root} \}$.
Then, by applying Lemma~8.1 to (\ref{eqn:8.3}), 
the definition  of $\mu_0(f)$ in (\ref{eqn:2.15})
gives the following. 

\begin{proposition}
$\mu_0(f)=\max\{
m_j:j\in\tilde{\mathcal R}(\overline{f})\}.$
\end{proposition}

\subsection{Invariance under the change of coordinates}

Let $(\Phi(x,y), \Psi(x,y))$ be a local diffeomorphism 
defined near the origin in $\R^2$, where 
$\Phi(x,y), \Psi(x,y)\in \R C^{\infty}((x,y))$ 
with $\Phi(0,0)=\Psi(0,0)=0$. 
Let 
$\overline{\Phi}(x,y),\overline{\Psi}(x,y)\in\R[[x,y]]$
be the Taylor series of $\Phi(x,y), \Psi(x,y)$, respectively. 

\begin{lemma}
Suppose that $\overline{P}(x,y),\overline{Q}(x,y)
\in \C[[x,y]]$ are irreducible and satisfy 
$\overline{P}(\overline{\Phi}(x,y),\overline{\Psi}(x,y))=
\overline{Q}(x,y)$
for the above $\overline{\Phi}(x,y),\overline{\Psi}(x,y)$. 
Then, the existence of a real root of $\overline{P}$
is equivalent to  that of $\overline{Q}$.
\end{lemma}

\begin{proof}
There exist real numbers 
$A,B,C,D$ with 
$\det 
\bigl(
\begin{smallmatrix}
   A & B \\
   C & D
\end{smallmatrix}
\bigl)\neq 0$ such that 
$\overline{\Phi}(x,y),\overline{\Psi}(x,y)\in\R[[x,y]]$
take a form
\begin{equation}\label{eqn:8.4}
\overline{\Phi}(x,y)=Ax+By+\cdots, \quad
\overline{\Psi}(x,y)=Cx+Dy+\cdots.
\end{equation}
If $\overline{\phi}(t),\overline{\psi}(t)\in\R[[t]]$ 
satisfy 
$\overline{\Phi}(\overline{\phi}(t),\overline{\psi}(t))=
\overline{\Psi}(\overline{\phi}(t),\overline{\psi}(t))=0$,
then we can see that
$\overline{\phi}(t)=\overline{\psi}(t)=\overline{0}$
by using $\det 
\bigl(
\begin{smallmatrix}
   A & B \\
   C & D
\end{smallmatrix}
\bigl)\neq 0$.
From this implication, the equivalence in the lemma
can be easily seen.
\end{proof}
Let us show that 
the quantity $\mu_0(f)$ is invariant under the change of coordinates.

\begin{proposition}
Let $(\Phi(x,y),\Psi(x,y))$ be a diffeomorphism as above and set
$g(x,y):=f(\Phi(x,y),\Psi(x,y))$. 
Then, we have $\mu_0(f)=\mu_0(g)$.
\end{proposition}

\begin{proof}
Let $\overline{f}, \overline{g}\in\R C^{\infty}((x,y))$ 
be the formal Taylor series of $f, g$ and 
let 
$\overline{f}$ admit the factorization (\ref{eqn:8.3}). 
Since $\overline{g}(x,y):=
\overline{f}(\overline{\Phi}(x,y),
\overline{\Psi}(x,y))$, 
$\overline{g}(x,y)$ can be represented as 
\begin{equation}\label{eqn:8.5}
\overline{g}(x,y)=
\prod_{j=1}^l \overline{P}_j
(\overline{\Phi}(x,y),\overline{\Psi}(x,y))^{m_j},
\end{equation}
where $\overline{P}_j(x,y)$ and $m_j$ are 
the same as in (\ref{eqn:8.3}).
Applying Lemmas 8.1 and 8.3 to (\ref{eqn:8.5}), 
we have 
$\mu_0(g)=\max\{m_j:j\in\tilde{\mathcal R}(\overline{f})\}$,
which implies $\mu_0(f)=\mu_0(g)$ from Proposition~8.2.  
\end{proof}


\subsection{From the viewpoint of Newton data}

Let us investigate properties of 
$\mu_0(f)$ by using the Newton data of 
$f\in\R C^{\infty}((x,y))$ in Section~2.1.
Hereafter, let $\overline{f}(x,y)=\sum_{j,k} c_{jk}x^j y^k$ be
the formal Taylor series of $f$.

A {\it face} of the Newton polygon $\Gamma_+(f)$ 
means an edge or a vertex 
of $\Gamma_+(f)$. 
For a compact face $\kappa$ of $\Gamma_+(f)$, 
the $\kappa$-{\it part} of $f$ is a polynomial defined by 
$f_{\kappa}(x,y)=\sum_{(j,k)\in\kappa} c_{jk} x^j y^k$.
We say
\begin{itemize}
\item
$f$ is {\it convenient} if 
$\Gamma_+(f)$ intersects every coordinate axis; 
\item
$f$ is $\R$-{\it nondegenerate} 
if $\nabla f_{\kappa} \neq 0$ holds on $(\R-\{O\})^2$
for every compact edge $\kappa$ of $\Gamma_+(f)$.
\end{itemize}
The above two conditions depend on the choice of coordinates. 
The $\R$-nondegeneracy condition was introduced 
by Kouchnirenko (see \cite{AGV88}) and 
plays useful roles in the study of singularity theory.
The above two conditions are related to the case of
$\mu_0(f)=0, 1$.
%
\begin{lemma}
\begin{enumerate}
\item
If $\mu_0(f)=0$, then $f$ is convenient on any coordinates. 
\item
If $f$ is convenient and $\R$-nondegenerate on some coordinate, 
then $\mu_0(f)=0$ or $1$.
\end{enumerate}
\end{lemma}


\begin{proof}
(i)\quad 
If $f$ is not convenient on some coordinate,
then  $\overline{f}(x,y)$ is divisible by $x$ or $y$
on the coordinate, 
which implies $\mu_0(f)\geq 1$.

(ii)\quad 
When $f$ is convenient, 
$\overline{f}$ admits the factorization 
of the form
$$
\overline{f}(t^N,y)=
\overline{u}(t^N,y)  
\prod_{j=1}^r
(y-\overline{\phi}_j(t))^{m_j},
$$
where $\overline{u}(x,y)\in\R[[x,y]]$, $m_j, N\in\N$, 
$\overline{\phi}_j(t)\in\C[[t]]$
are the same as in (\ref{eqn:2.13}) and, moreover,  
$\overline{\phi}_j\neq \overline{0}$ for all $j$.
Now, we assume that $\mu_0(f)\geq 2$. 
Then, there exists $j\in\{1,\ldots,r\}$ such that
$m_j\geq 2$ and $\overline{\phi}_j(t)\in\R [[t]]$
takes the form 
$\overline{\phi}_j(t)=ct^n+\cdots$ 
where $c\neq 0$ and $n\in\N$. 
Let $\kappa$ be the edge of $\Gamma_+(f)$
which is contained in the line 
$\{(j,k)\in\R_+^2:Nj+nk=L\}$
where $L$ is some positive number. 
It is easy to see that 
$f_{\kappa}(t^N,y)$ is divisible by
$(y-ct^n)^2$, 
which is a contradiction 
to the $\R$-nondegeneracy condition of $f$. 
\end{proof}

In Theorem~2.5, 
the case of $\mu_0(f)\geq 2$ is particularly interesting 
in the $C^{\infty}$ case. 
Therefore, it suffices to consider the case where
$f$ is not convenient or does not satisfy 
the $\R$-nondegeneracy condition
from the above lemma.

\begin{remark}
The converses of the implications in the above lemma 
are not true. 
In the case of (i), consider the example $f(x,y)=y^2-x^3$. 
We can see that $\mu_0(f)=1$ but $f$ is always convenient 
on any coordinate. 
In the case of (ii), 
consider the example $f(x,y)=xy$. 
We can see that $\mu_0(f)=1$ but $f$ is not convenient. 
\end{remark}

Let us consider the relationship 
between the invariant $\mu_0(f)$ and 
the height $\delta_0(f)$ defined in Section~2.1. 
For this purpose, we need a result in \cite{IkM11tams},  
which gives a useful equivalence condition to the ``adaptedness''
of coordinates (see Section~2.1).
The minimal face of the Newton polygon 
$\Gamma_+(f)$ containing 
the point $(d(f),d(f))$ is called the {\it principal face}
of $\Gamma_+(f)$, 
which is denoted by $\kappa_P$. 
When $\kappa_P$ is compact, 
the $\kappa_P$-part of $f$ is called 
the {\it principal part} of $f$, which is denoted 
by $f_P$.

\begin{lemma}[\cite{AGV88}, \cite{IkM11tams}]
The coordinates are adapted to $f$ if and only if
one of the following conditions is satisfied:
\begin{enumerate}
\item $\kappa_P$ is a compact edge 
and 
$\mu_0(f_P)\leq d(f)$;
\item $\kappa_P$ consists of a vertex;
\item $\kappa_P$ is unbounded.
\end{enumerate}
Moreover, in case (i) we have $\delta_0(f)=\delta_0(f_P)=d(f_P)$. 
\end{lemma}

The following proposition shows that
$\mu_0(f)$  also indicates 
some kind of flatness of $f$ at the orign.

\begin{proposition}
$\mu_0(f)\leq \delta_0(f)$ holds for 
every non-flat $f\in \R C^{\infty}((x,y))$.
\end{proposition}

\begin{proof}
When $\kappa_P$ is a compact edge, 
the definition of $\mu_0(f)$ directly gives $\mu_0(f)\leq \mu_0(f_P)$, 
which implies 
$\mu_0(f)\leq \delta_0(f)$ by using Lemma~8.7.
It is easy to show the estimate 
in the case where $\kappa_P$ satisfies (ii), (iii) in Lemma~8.7.

\end{proof}

\begin{remark}
The estimate in the proposition also holds 
even if $f$ does not satisfy the condition (\ref{eqn:2.1}).
\end{remark}

From Theorems 2.2 and 2.5,  
the difference
of $\mu_0(f)$ and $\delta_0(f)$ is important
in the analytic continuation issue 
for local zeta functions. 
Let us consider when 
an equal sign is established in the estimate in Proposition~8.8. 

(i)\quad 
When $\kappa_P$ is unbounded, 
it is easy to see that $\mu_0(f)=\delta_0(f)$ always holds.

(ii)\quad
When $\kappa_P$ is a vertex, 
both cases are possible. 
For example, consider the functions
$f_1(x,y)=x^2 y^2$ and 
$f_2(x,y)=x^6+x^2 y^2+y^6$. 
We can see that $\mu_0(f_1)=\delta(f_1)=2$ and 
$\mu_0(f_2)=0$, $\delta_0(f_2)=2$.

(iii)\quad 
When $\kappa_P$ is a compact edge, 
both cases are also possible. 
For example, consider the functions
$f_1(x,y)=(x^2-y^2)^2$ and 
$f_2(x,y)=x^4+y^4$. 
We can see that 
$\mu_0(f_1)=\delta(f_1)=2$ and 
$\mu_0(f_2)=0$, $\delta_0(f_2)=2$. 
Notice that the Newton polygons of the above examples are the same. 
Altough the shapes of the Newton polygons do not always 
determine the establishment of the equal sign, 
the following lemma gives the necessary condition 
on the Newton polygon for this establishment. 

For $\kappa_P$, 
there exists a unique pair $(l_1,l_2)$ with $l_1,l_2>0$ such that
$\kappa_P$ is contained in the line 
$\{(\alpha,\beta)\in\R_+^2:l_1 \alpha+ l_2 \beta=1\}$. 
Without loss of generality, by flipping coordinates, if necessary, 
we may assume that $l_2\geq l_1$. 

\begin{lemma}
If $\mu_0(f)=\delta_0(f)$, then
$l_2/l_1$ is a positive integer. 
\end{lemma} 

\begin{proof}
If $l_2/l_1$ is not a positive integer, 
then it is shown in \cite{IkM11tams} that 
$\mu_0(f_P)<d(f)$ holds, which 
shows the adaptedness by Lemma~8.6 (i)
(i.e., $d(f)=\delta_0(f)$) and 
implies that 
$\mu_0(f)<\delta_0(f)$ holds.
\end{proof}

\begin{remark}
A typical example satisfying the assumption of the lemma 
is $f(x,y)=y^2-x^3$. 
We can see that 
$\mu_0(f)=1$ and $\delta_0(f)=6/5$.
\end{remark}

%


\section{Meromorphy of local zeta functions in model cases}

Sections~9-13 are the analytic part of this paper. 
By using almost resolution of singularities in Theorem~3.2, 
we will investigate the analytic continuation of 
the local zeta function which is defined by the integral
\begin{equation}\label{eqn:9.1}
Z(f,\varphi)(s)=
\int_{\R^2}|f(x,y)|^s \varphi(x,y)dxdy 
\,\,\,\,\,\quad s\in \C,
\end{equation}
where, $f, \varphi$ are as 
in the beginning of the Introduction. 
Moreover, we assume that
$f$ is non-flat at the origin and 
satisfies the condition (\ref{eqn:2.1}). 
Since the situation of analytic continuation of (\ref{eqn:9.1}) 
does not change under a smooth change of integral variables,  
we may assume that $f(x,y)$ is regular in $y$ of order $n$
where $n$ is a positive integer
(i.e., $f(0,y)=cy^n+\cdots$ with $c\neq 0$).

It is easy to see that 
the integral $Z(f,\varphi)(s)$ can be decomposed as 
$$Z(f,\varphi)(s)=Z_+(s)+Z_-(s),$$
where  
\begin{equation}\label{eqn:9.2}
Z_{\pm}(s)=\int_{\R}\int_0^{\infty}
|f(\pm x,y)|^s 
\varphi(\pm x, y) dx dy.
\end{equation}
Since properties of $Z_+(s)$ and $Z_-(s)$ 
are essentially the same, 
we will only deal with the case of $Z_+(s)$. 
Furthermore, 
we may assume that the Taylor series of $f$ admits the factorization
(\ref{eqn:2.13}) in Section~2 and  
let $N$ be a positive even integer in (\ref{eqn:2.13}).
By simple change of an integral variable,
\begin{equation}\label{eqn:9.3}
\begin{split}
Z_+(s)&=
N \int_{\R}\int_0^{\infty}
|f(x^N,y)|^s 
\varphi(x^N, y) x^{N-1}dx dy \\
&=\frac{N}{2} \int_{\R^2}
|F(x,y)|^s 
\Phi(x, y) x^{N-1}dx dy,
\end{split}
\end{equation}
where $F(x,y)=f(x^N,y)$ and $\Phi(x,y)=\varphi(x^N,y)$.
From the definition of $F$, 
we may assume that 
the Taylor series $\overline{F}(x,y)$ of $F(x,y)$
admits a factorization of the form 
(\ref{eqn:3.1}) in Section~3. 
From the relationship between $f(x,y)$ and $F(x,y)$, 
we can easily see $\mu_0(f)=\mu_0(F)$. 

\subsection{Decomposition of integrals}
It follows from Theorem~3.2 that
there exists a $C^{\infty}$ manifold $Y$ and an almost resolution 
of singularities $\pi:Y\to U$ for a $C^{\infty}$ function $F$. 
Here, an open set $U$, on which $F$ and $\Phi$ are defined, 
must be taken to be so small that 
the above resolution exists. 
Let us prepare a smooth partition of unity
on $Y$ as follows.  
There exist $C^{\infty}$ functions 
$\chi_j$, for $j=1,\ldots, r$, and 
$\hat{\chi}_k$, for $k=1,\ldots, \hat{r}$,
defined on $Y$ such that the following properties:
\begin{itemize}
\item $\chi_j$ and $\hat{\chi}_k$ are $C^{\infty}$
functions on $Y$ with small supports;
\item
The support of each $\chi_j$ contains an open neighborhood of 
the point $P_j$, which is as in Theorem~3.2; 
\item
The support of each $\hat{\chi}_k$ does not contain 
the set $\{P_1,\ldots,P_r\}$;
\item $\sum_{j=1}^{r} \chi_j(x,y)+
\sum_{k=1}^{\hat{r}} \hat{\chi}_k(x,y)=1$ 
holds on $Y$.
\end{itemize}

The integrand in (\ref{eqn:9.3}) is denoted by 
$A(x,y;s)$, that is,
\begin{equation*}
A(x,y;s):=\frac{N}{2} |F(x,y)|^s \Phi(x, y) x^{N-1}.
\end{equation*}
Then $Z_+(s)$ can be expressed as 
\begin{equation}\label{eqn:9.4}
Z_+(s):=\int_{Y} 
A(\pi(x,y);s)|J_{\pi}(x,y)|dxdy,
\end{equation}
where $J_{\pi}$ is the Jacobian of $\pi$.
Furthermore, 
by using the cut-off functions 
$\chi_j$ and $\hat{\chi}_k$, 
we have 
\begin{equation}\label{eqn:9.5}
Z_+(s)=\sum_{j=1}^{r} Z_j(s) +
\sum_{k=1}^{\hat{r}} \hat{Z}_k(s),
\end{equation}
where 
\begin{equation}\label{eqn:9.6}
\begin{split}
&Z_j(s):=\int_{\R^2} 
A(\pi(x,y);s)|J_{\pi}(x,y)|\chi_j(x,y)dxdy; \\
&\hat{Z}_k(s):=\int_{\R^2} 
A(\pi(x,y);s)|J_{\pi}(x,y)|\hat{\chi}_k(x,y)dxdy.
\end{split}
\end{equation}

It follows from Theorem~3.2 (ii) that
every $Z_j(s)$ can be expressed as the sum
of the integrals of the form:
\begin{equation}\label{eqn:9.7}
{\mathcal Z}(s)=\int_{\R_+^2}
|G(x,y)|^s x^b \phi(x,y)dxdy,
\end{equation}
where $b$ is a nonnegative integer, 
$\phi$ is a $C^{\infty}$ function defined near the origin 
whose support is sufficiently small and  
\begin{equation}\label{eqn:9.8}
G(x,y):=u(x,y)x^a
\left(y^m+
\varepsilon_1(x)y^{m-1}+\cdots+\varepsilon_m(x)
\right),
\end{equation}
where $a, m$ are positive integers and $\varepsilon_j(x)$ 
are real-valued flat $C^{\infty}$ functions defined near the origin.
Note that $m$ belongs to $\{m_j:j\in{\mathcal R}(F)\}$, 
where $m_j$ are as in (\ref{eqn:3.1}) and 
${\mathcal R}(F)$ is as in (\ref{eqn:2.14}).
From the convergence of integral, 
${\mathcal Z}(s)$ can be regarded as a holomorphic function 
on the region ${\rm Re}(s)>0$.
The following theorem shows 
more precise situation of analytic continuation, 
which is the most important results 
in this paper from an analytical point of view.  

\begin{theorem}
${\mathcal Z}(s)$ admits the meromorphic continuation to the region 
${\rm Re}(s)> -1/m$. 
Furthermore,  when $m\geq a/b$, 
the poles of the extended ${\mathcal Z}(s)$ 
on ${\rm Re}(s)>-1/m$ exist in the set 
$\{-(b+j)/a: j \in\N\}$. 
\end{theorem}

After investigating the properties of ${\mathcal Z}(s)$ 
in Sections 10--12, 
we will prove the above theorem  in Section~13.

On the other hand, 
it follows from Theorem~3.2 (iii) that
every $\hat{Z}_k(s)$ can be expressed as the sum of the integrals 
of the form:
\begin{equation}\label{eqn:9.9}
\hat{{\mathcal Z}}(s)=\int\int_{\R_+^2}
x^{as+c} y^{bs+d} v(x,y)^s \phi(x,y)dxdy,
\end{equation}
where $a,b,c,d$ are nonnegative integers, 
$\phi(x,y)$ is as in (\ref{eqn:9.7}) and
$v(x,y)$ is a positive $C^{\infty}$ function defined near the origin.
From an elementary analysis in Lemma~11.1, below, 
we can see that $\hat{{\mathcal Z}}(s)$ admits 
the meromorphic continuation to 
the whole complex plane and its poles exist in 
the set 
$\{-(c+j)/a, -(d+k)/b: j,k\in\N\}$.


\section{Proof of Theorem~2.5}

In the case of $\mu_0(f)=0,1$, 
an ordinary resolution of singularities for $f$ 
can be constructed (see Remark~3.3),
which implies ${\mathfrak m}_0(f)=\infty$
by an elementary analysis in Lemma~11.1. 

In the case of $\mu_0(f)\geq 2$, 
we may only deal with the case of $Z_+(s)$ 
in (\ref{eqn:9.2}). 
By using almost resolution of singularities for $F$
in (\ref{eqn:9.3})
and an appropriate smooth partition of unity, 
$Z_+(s)$ can be expressed as 
a finite sum of the integrals 
$Z_j(s)$ and $\hat{Z}_k(s)$ as in (\ref{eqn:9.5}), 
where each $Z_j(s)$ (resp. $\hat{Z}_k(s)$) 
can be expressed as the sum of the integrals of the form
${\mathcal Z}(s)$ in (\ref{eqn:9.7}) 
(resp.  $\hat{\mathcal Z}(s)$ in (\ref{eqn:9.9})), 
which is explained in Section~9. 
From the form of $\hat{\mathcal Z}(s)$, 
$\hat{Z}_k(s)$ can always admit the meromorphic continuation
on the whole complex plane. 
On the other hand, 
since $m$ belongs to $\{m_j:j\in{\mathcal R}(F)\}$
where $m_j$ are as in (\ref{eqn:2.13}) and 
${\mathcal R}(F)$ is as in (\ref{eqn:2.14}), 
every $Z_j(s)$ admits the meromorphic 
continuation to the region 
${\rm Re}(s)>-\min\{1/m_j:j\in{\mathcal R}(f) \}
=-1/\mu_0(F)$ from Theorem~9.1.
Moreover, since
the relationship between $f(x,y)$ and $F(x,y)$
implies $\mu_0(f)=\mu_0(F)$, 
the claim (ii) in the theorem can be shown. 
We remark that a coordinate is chosen
so that $f(x,y)$ is regular in $y$ in the beginning 
of Section~9, 
which gives no essential influence on the above discussion
by using Proposition 8.4.

The positions of candidate poles of the integrals
${Z}_j(s)$ and $\hat{Z}_k(s)$
have been explained in Section~9, which 
implies the latter part of the theorem.

\section{Auxiliary lemmas}


\subsection{Meromorphy of one-dimensional model}

The following lemma is essentially known
(see \cite{GeS64}, \cite{BeG69}, \cite{AGV88}, etc.).
Since we will use
not only the result but also an idea of its proof
in the later computation, 
we give a complete proof here.

\begin{lemma}\label{lem:2.2}
Let $\psi: [0,r] \times \C\to\C$, with $r>0$, 
satisfy the following properties:
\begin{enumerate}
\item[(a)]
$\psi(\cdot;s)$ is smooth on $[0,r]$ 
for all $s\in \C$;
\item[(b)] 
$\dfrac{\d^{\a} \psi}{\d u^{\a}}(u;\cdot)$ is
an entire function on $\C$ for all $u\in [0,r]$ and $\a\in \Z_+$.
\end{enumerate}
Let $L(s)$ be an integral defined by 
\begin{equation}\label{eqn:11.1}
L(s):=\int_{0}^{r}u^{A s+B}\psi(u;s)du \quad\quad s\in\C,
\end{equation}
where $A$ is a positive integer and 
$B$ is a nonnegative integer. 

Then the following hold.
\begin{enumerate}
\item
The integral
$L(s)$ 
becomes a holomorphic function on the half-plane ${\rm Re}(s)>-(B+1)/A$, 
which is also denoted by $L(s)$.
\item
Furthermore, 
$L(s)$ admits the meromorphic continuation 
to the whole complex plane. 
Moreover, its poles are simple and they exist in the set 
$\{-(B+j)/A: j\in \N\}$.
\end{enumerate}
\end{lemma}

\begin{proof} \quad 
(i) \quad 
Since the integral $L(s)$ locally uniformly converges 
on the half-plane ${\rm Re}(s)>-(B+1)/A$,
the assumption and the Lebesgue convergence theorem imply that
the integral becomes a holomorphic function there.

\vspace{.5 em}

(ii) \quad
Let $N$ be an arbitrary natural number.
The Taylor formula implies
\begin{equation}\label{eqn:11.2}
\psi(u;s)=\sum_{\a=0}^{N}\frac{1}{\a !}
\frac{\d^{\a} \psi}{\d u^{\a}}(0;s)u^{\a}+u^{N+1}R_N(u;s)
\end{equation}
with
\begin{equation*}
R_N(u;s)=
\frac{1}{N!}\int_{0}^{1}(1-t)^{N}
\frac{\d^{N+1} \psi}{\d u^{N+1}}(tu;s)dt.
\end{equation*}
Here $R_N(u;s)$ satisfies the following.    
\begin{itemize}
\item
$R_N(\cdot;s)$ is smooth on $[0,r]$ 
for all $s\in \C$.
\item
$R_N(u;\cdot)$ is
an entire function on $\C$ for all $u\in [0,r]$.
\end{itemize}.
Substituting (\ref{eqn:11.2}) into the integral (\ref{eqn:11.1}), 
we have 
\begin{equation}\label{eqn:11.3}
\begin{split}
L(s)=
\sum_{\a=0}^{N}
\frac{r^{As+B+\a+1}}{\a!(As+B+\a+1)}
\frac{\d^{\a} \psi}{\d u^{\a}}(0;s)+
\int_{0}^{r}u^{As+B+N+1}R_N(u;s)du
\end{split}
\end{equation}
on ${\rm Re}(s)>-(B+1)/A$.
From (i), the integral in (\ref{eqn:11.3}) 
becomes a holomorphic function on the half-plane ${\rm Re}(s)>-(B+N+2)/A$.
Therefore, 
$L(s)$ can be analytically continued as a meromorphic function
to the half-plane ${\rm Re}(s)>-(B+N+2)/A$.  
Moreover, all the poles of the above meromorphic function are simple 
and they are contained in the set 
$\{-(B+j)/A: j=1,\ldots, N+1\}$.
Letting $N$ tend to infinity, we have the assertion.
\end{proof}

\subsection{Meromorphy of important integrals}

Let $p$ be a positive even integer and 
let $r, R$ be positive real numbers 
satisfying $r^p\leq R$.
Let $U$ be an open neighborhood of 
the set 
$\{(u,v)\in\R^2:|u|\leq r, \, |v|\leq R\}$.
We define
\begin{equation}\label{eqn:11.4}
D_p:=\{(u,v)\in U: 
|u| \leq r, \, \, u^p < v \leq R \}.
\end{equation}

The following lemma will play a useful 
role in the proof of Theorem~9.1.

\begin{lemma}
Let $\Phi:U \times \C\to\C$ satisfy 
the following properties.
\begin{enumerate}
\item[(a)]
$\Phi(\cdot; s)$ is a $C^{\infty}$ function on $U$
for all $s\in \C$;
\item[(b)]
$\dfrac{\d^{\alpha+\beta}\Phi}{\d u^{\alpha}\d v^{\beta}}(u,v;\cdot)$ 
is an entire function 
for all $(u,v)\in D_p$ and $(\alpha, \beta)\in \Z_+^2$.
\end{enumerate}
Let $H(s)$ be an integral defined by
\begin{equation}\label{eqn:11.5}
H(s):=\int_{D_p\cap \R_+^2}
u^{as+b}v^{ms}\Phi(u,v;s)dudv \quad \,\,
s\in\C,
\end{equation}
where $a, m$ are positive integers and 
$b$ is a nonnegative integer. 

Then the following hold.
\begin{enumerate}
\item
The integral $H(s)$ becomes a holomorphic function on the half-plane: 
$$
{\rm Re}(s)>
\max\left\{-\frac{b+1}{a}, -\frac{1}{m}\right\},
$$
which is also denoted by $H(s)$.
\item
Furthermore, $H(s)$ admits the meromorphic continuation
to the whole complex plane and 
its poles exist in the set
\begin{equation}\label{eqn:11.6}
\left\{-\frac{b+j}{a},-\frac{k}{m},-\frac{b+j+pk}{a+pm}:
j,k\in \N\right\}.
\end{equation}
\end{enumerate}
\end{lemma}

\begin{remark}
In order to see the distribution of poles of $H(s)$, 
observe the case 
when $j=k=1$ in (\ref{eqn:11.6}).
It is easy to see  
\begin{equation*}
\begin{split}
&\min
\left\{-\frac{b+1}{a},-\frac{1}{m}\right\}
\leq-\frac{b+p+1}{a+pm}\leq
\max
\left\{-\frac{b+1}{a},-\frac{1}{m}\right\}, \\
&\quad\quad\quad\quad 
\lim_{p\to\infty}
\frac{b+p+1}{a+pm}=\frac{1}{m}.
\end{split}
\end{equation*}
\end{remark}

\begin{proof}[Proof of Lemma~11.2]
(i)\quad
In a similar fashion to the proof of the Lemma~11.1 (i),
it can be easily shown that the integral $H(s)$ becomes a holomorphic function
defined on the half-plane 
${\rm Re}(s)>
\max\left\{-(b+1)/a, -1/m\right\}$. 
\vspace{.5 em}

(ii)\quad
Let us consider the meromorphic continuation of $H(s)$.
For simplicity,
we will use the following symbol:
For each $(\alpha,\beta)\in \Z_+^2$,
\begin{equation*}
\Phi^{\langle \alpha,\beta \rangle}(u,v;s):=
\frac{\d^{\alpha+\beta}\Phi}{\d u^{\alpha}\d v^{\beta}}(u,v;s).
\end{equation*}
Let $N$ be an arbitrary natural number.
By the Taylor formula,
\begin{equation}\label{eqn:11.7}
\begin{split}
\Phi(u,v;s)
&=\sum_{(\a,\b)\in \{0,\ldots,N\}^2}
\frac{\Phi^{\langle \a, \b \rangle}(0,0;s)}{\a ! \b !}
u^{\a}v^{\b}
+
\sum_{\a=0}^{N}u^{\a}v^{N+1}
\tilde{A}_{\alpha}^{(N)}(v;s)\\
&\quad\quad\quad
+\sum_{\b=0}^{N}u^{N+1}v^{\b}
\tilde{B}_{\beta}^{(N)}(u;s)
+u^{N+1}v^{N+1}\tilde{C}^{(N)}(u,v;s),
\end{split}
\end{equation}
where
\begin{equation}\label{eqn:11.8}
\begin{split}
&\tilde{A}_{\alpha}^{(N)}(v;s)
:=\frac{1}{\a ! N!}\int_{0}^{1}(1-t)^{N}
\Phi^{\langle \a,N+1 \rangle}(0,tv;s)dt
 \mbox{\quad for } \a\in\{0,\ldots, N\}, \\
&\tilde{B}_{\beta}^{(N)}(u;s)
:=\frac{1}{\b ! N!}
\int_{0}^{1}(1-t)^{N}
\Phi^{\langle N+1, \b \rangle}(tu,0;s)dt
\mbox{\quad for }  \b\in \{0,\ldots,N\},\\
&\tilde{C}^{(N)}(u,v;s) \\
&\quad:=\frac{1}{(N!)^2}\int_{0}^{1}\int_{0}^{1}
(1-t_1)^{N}(1-t_2)^{N}
\Phi^{\langle N+1, N+1 \rangle}(t_1u,t_2v;s)
dt_1 dt_2.
\end{split}
\end{equation}
Note that
\begin{itemize}
\item 
$\tilde{A}_{\alpha}^{(N)}(\cdot;s)$ 
is a $C^{\infty}$ function on $[-R,R]$ for all $s\in\C$,
\item 
$\tilde{B}_{\beta}^{(N)}(\cdot;s)$
is a $C^{\infty}$ function on $[-r,r]$ for all $s\in\C$,
\item 
$\tilde{C}^{(N)}(\cdot;s)$ 
is a $C^{\infty}$ function on $U$ for all $s\in\C$,
\item 
$\tilde{A}_{\alpha}^{(N)}(v;\cdot)$, 
$\tilde{B}_{\beta}^{(N)}(u;\cdot)$, 
$\tilde{C}^{(N)}(u,v;\cdot)$ 
are entire functions for all $(u,v)\in U$.
\end{itemize}
Substituting (\ref{eqn:11.7}) into (\ref{eqn:11.5}), 
we have 
\begin{equation}\label{eqn:11.9}
\begin{split}
H(s)=&\sum_{(\a,\b)\in \{0,\ldots,N\}^2}
\frac{\psi^{\langle \a, \b \rangle}(0,0;s)}{\a ! \b !}H_{\a, \b}(s)\\
&\quad+
\sum_{\alpha=0}^N A_{\alpha}^{(N)}(v;s)+
\sum_{\beta=0}^N B_{\beta}^{(N)}(v;s)+
C^{(N)}(s)
\end{split}
\end{equation}
with
\begin{equation*}
\begin{split}
&H_{\a, \b}(s)=
\int_{D_p\cap \R_+^2}
u^{as+b+\a}v^{ms+\b}dudv 
\mbox{\quad for } (\a, \b)\in \{0,\ldots, N\}^2, \\
&A_{\alpha}^{(N)}(s)=
\int_{D_p\cap \R_+^2}
u^{as+b+\a}v^{ms+N+1}\tilde{A}_{\alpha}^{(N)}(v;s)dudv 
 \mbox{\quad for } \a\in\{0,\ldots, N\},\\
&B_{\beta}^{(N)}(s)=
\int_{D_p\cap \R_+^2}
u^{as+b+N+1}v^{ms+\b}\tilde{B}_{\beta}^{(N)}(u;s)dudv  
\mbox{\quad for }  \b\in \{0,\ldots,N\},\\
&C^{(N)}(s)=
\int_{D_p\cap \R_+^2}
u^{as+b+N+1}v^{ms+N+1}\tilde{C}^{(N)}(u,v;s)dudv.
\end{split}
\end{equation*}


Now let us consider the meromorphy of the above 
integrals. 

\vspace{.5 em}

\clearpage

\underline{The integral $H_{\a,\b}(s)$.}\quad 

A simple calculation gives  
\begin{equation*}
\begin{split}
&H_{\a,\b}(s) 
=
\int_{0}^{r}u^{as+b+\a}
\left(\int_{u^p}^{R}v^{ms+\b}dv\right)du
\\
&=
\frac{1}{ms+\beta+1}
\left( 
\frac{R^{ms+\beta+1}r^{as+b+\alpha+1}}{as+b+\alpha+1}
-\frac{r^{(a+mp)s+\a+b+p(\beta+1)+1}}
{(a+mp)s+\a+b+p(\beta+1)+1}
\right),
\end{split}
\end{equation*}
which implies that 
every $H_{\a,\b}(s)$ becomes a meromorphic function 
on the whole complex plane
and has three poles which are contained in the set
\begin{equation}\label{eqn:11.10}
\left\{-\frac{b+j}{a},-\frac{k}{m},-\frac{b+j+pk}{a+mp}:
j,k\in\N \right\}.
\end{equation}

\vspace{.5 em}

\underline{The integral $A_{\a}^{(N)}(s)$.}\quad 

A simple calculation gives  
\begin{eqnarray}
&&A_{\a}^{(N)}(s)=
\left(
\int_{0}^{r^p}\int_0^{v^{1/p}}+
\int_{r^p}^{R}\int_0^{r}
\right)
u^{as+b+\a}v^{ms+N+1}
\tilde{A}_{\a}^{(N)}(v;s)dudv \nonumber\\
&&\quad\quad\quad
\begin{split}
&=\frac{1}{as+b+\a+1}
\left(
\int_{0}^{r^p}
u^{\frac{1}{p}\{(a+mp)s+b+pN+p+\a+1\}}
\tilde{A}_{\a}^{(N)}(v;s)dv \right.
\\
&\quad\quad\quad\quad\quad\quad\quad\quad\quad\quad 
\left. +
r^{as+b+\alpha+1}\int_{r^p}^{R}
v^{ms+N+1}
\tilde{A}_{\a}^{(N)}(v;s)dv
\right). \label{eqn:11.11}
\end{split}
\end{eqnarray}
Lemma~11.1 (i) implies that
the first integral in (\ref{eqn:11.11})
becomes a holomorphic function 
on the half-plane 
${\rm Re}(s)>-(pN+2p+b+\a+1)/(a+mp)$.
Moreover, it is easy to check that the second integral
is an entire function.  
Hence, $A_{\alpha}^{(N)}(s)$ 
can be analytically continued 
as a meromorphic function to the half-plane 
${\rm Re}(s)>-(pN+2p+b+\a+1)/(a+mp)$
and has one pole which is contained in the set
\begin{equation}\label{eqn:11.12}
\left\{-\frac{b+j}{a}:j\in\N \right\}.
\end{equation}

\vspace{.5 em}

\underline{The integral $B_{\beta}^{(N)}(s)$.}\quad

A simple calculation gives 
\begin{eqnarray}
&&B_{\beta}^{(N)}(s)=
\int_0^r\left(\int_{u^p}^{R}
v^{ms+\b}dv
\right)
u^{as+b+N+1}
\tilde{B}_{\b}^{(N)}(u;s)du \nonumber\\
&&
\begin{split}
&\quad\quad\quad
=\frac{1}{ms+\b+1}
\left(
R^{ms+\b+1}\int_{0}^{r}u^{as+b+N+1}
\tilde{B}_{\b}^{(N)}(u;s)du
\right. \\ 
&\quad\quad\quad\quad\quad\quad \quad\quad
\left.
-\int_{0}^{r}u^{(a+mp)s+b+p\b+p+N+1}
\tilde{B}_{\b}^{(N)}(u;s)du
\right). \label{eqn:11.13}
\end{split}
\end{eqnarray}
Lemma~11.1 (i) implies that
the first (resp. the second) integral in 
(\ref{eqn:11.13})
becomes a holomorphic function on the half-plane
${\rm Re}(s)>-(b+N+2)/a$ 
(resp. ${\rm Re}(s)>-(b+p\b+p+N+2)/(a+pm)$).
Thus, 
$B_{\beta}^{(N)}(s)$ can be analytically continued 
as a meromorphic function to the half-plane 
${\rm Re}(s)>
\max\{-(b+N+2)/a,-(b+p\b+p+N+2)/(a+pm)\}$ 
and has one pole which is contained in the set
\begin{equation}\label{eqn:11.14}
\left\{-\frac{k}{m}:k\in\N \right\}.
\end{equation}
(When $N$ is sufficiently large, 
the above maximum is $-(b+p\beta+p+N+2)/(a+pm)$.)

\vspace{.5 em}

\underline{The integral $C^{(N)}(s)$.}\quad 

It follows from the proof of (i) in this lemma that
the integral $C^{(N)}(s)$ converges
on the half-plane 
${\rm Re}(s)>\max\{-(b+N+2)/a,-(N+2)/b\}$,
which implies that
$C^{(N)}(s)$ can be analytically continued as a holomorphic function 
there.

\vspace{.5 em}

From the above, letting $N$ tend to infinity 
in (\ref{eqn:11.9}), 
we can see that $H(s)$ becomes a meromorphic function 
on the whole complex plane and that
its poles are contained in the set (\ref{eqn:11.6})
from (\ref{eqn:11.10}), (\ref{eqn:11.12}), (\ref{eqn:11.14}).
\end{proof}

%

\subsection{A van der Corput-type lemma}

\begin{lemma}[\cite{Gre06}]
Let $f$ be a $C^{k}$ function on an interval $I$ in $\R$.
If 
$|f^{(k)}|>\eta$
on $I$, then
for $\sigma\in (-1/k,0)$ there is a positive constant 
${\mathcal C}(\sigma,k)$ depending only on $\sigma$ and $k$
such that
\begin{equation}\label{eqn:11.15}
\int_{I}|f(x)|^{\sigma}dx<
{\mathcal C}(\sigma,k)
\eta^{\sigma}|I|^{1+k\sigma},
\end{equation}
where $|I|$ is the length of $I$.
\end{lemma}

The above van der Corput-type lemma plays the most important role 
in our analysis. 
This lemma has been shown in \cite{Gre06}, \cite{KaN20}.

\section{Analytic properties of $G(x,y)$ in (\ref{eqn:9.8})}

In this section, let us investigate
analytic properties of the $C^{\infty}$ function
\begin{equation*}
G(x,y)=u(x,y)x^a
\left(y^m+
\varepsilon_1(x)y^{m-1}+\cdots+\varepsilon_m(x)
\right),
\end{equation*}
where $a, m$ are positive integers, 
$u(x,y)\in C^{\infty}(U)$ satisfies $u(0,0)>0$ 
and $\varepsilon_j$ are real-valued flat $C^{\infty}$ functions
defined near the origin.


\subsection{The decisive curve $C_G$ and its branches}

In order to understand the analytic continuation
of the function ${\mathcal Z}(s)$, 
it is important to 
understand geometric properties of the decisive curve $C_G$ defined by $G$.
From the argument in Section~6, 
there exists a small open neighborhood $U$ of the origin 
such that 
\begin{equation}\label{eqn:12.1}
G(x,y)=u(x,y)x^a\prod_{k=1}^m (y-\gamma_k(x)) \,\,\,
\mbox{ on $U$},
\end{equation}
where each $\gamma_k\in C(U)$ satisfies that 
$\gamma_k(x)=O(x^l)$ for any $l\in\N$ and 
$u\in C^{\infty}(U)$ satisfies $u(0,0)>0$.

From (\ref{eqn:12.1}), 
the decisive curve $C_G$ is composed of ($m+1$) branches
parametrized as: for $t\in\R$ with small $|t|$,  
\begin{equation*}
\begin{split}
&B_0: \quad x=0, \,\, y=t, \\
&B_k: \quad x=t, \,\, 
y=\gamma_k(t) \,\,\text{ for $k=1,\ldots,m$.}
\end{split}
\end{equation*}

Roughly speaking, 
for the meromorphic extension of ${\mathcal Z}(s)$, 
the branch $B_0$ gives {\it good} influence;
while the branches $B_k$ 
with $\gamma_k\not\equiv 0$ 
may give {\it bad} one.
In the analysis of ${\mathcal Z}(s)$, 
the integral region is divided into two parts: 
one is avoided from the bad branches as large as possible, 
while the other is its complement. 
Of course, their shapes must be chosen so that 
they will be effective for the computation.
Actually, 
we use the two regions with a parameter $p$ which is 
a large even integer:
\begin{equation}\label{eqn:12.2}
\begin{split}
U_1^{(p)}&=\{(x,y)\in U: |x| \leq r_p, \,\, y>x^{p}\},\\
U_2^{(p)}&=\{(x,y)\in U: |x| \leq r_p, \,\, 0<y \leq x^{p}\},
\end{split}
\end{equation}
where $r_p>0$ will be appropriately decided later 
in (\ref{eqn:12.5}).

\subsection{Two expressions of $G$}

In order to understand 
important properties of $G$,
we express $G$ by using the two functions
$G_1:U\setminus\{y=0\}\to\R$ and $G_2:U\to\R$
defined by 
\begin{equation}\label{eqn:12.3}
\begin{split}
&G_1(x,y)=u(x,y)
\left(
1+\sum_{j=1}^{m}
\frac{\tilde{\varepsilon}_{m-j}(x)}{y^j}
\right),\\
&G_2(x,y)=u(x,y)
\left(
y^m+\sum_{j=0}^{m-1}
\tilde{\varepsilon}_j(x) y^j
\right),
\end{split}
\end{equation}
where 
$\tilde{\varepsilon}_j(x):=\varepsilon_{j}(x)/x^a$ 
for $j=0,\ldots,m-1$. 
Note that each $\tilde{\varepsilon}_j$ is a flat $C^{\infty}$ function 
defined near the origin. 
Then $G$ can be expressed as  
\begin{equation}\label{eqn:12.4}
\begin{split}
G(x,y)&=x^a y^m G_1(x,y) \quad 
\mbox{ on $U\setminus\{y=0\}$,}\\
G(x,y)&=x^a G_2(x,y) \quad \mbox{ on $U$.}
\end{split}
\end{equation}

In order to investigate $G$, we use 
$G_1$ on $U_1^{(p)}$ and
$G_2$ on $U_2^{(p)}$.


\subsection{Properties of $G_1$}

Let $\gamma$ be the real-valued function defined near the origin by
$$\gamma(x)=\max\{|\gamma_k(x)|:k=1,\ldots,a\},$$
where $\gamma_k$ are as in (\ref{eqn:12.1}).
Since each $\gamma_k(x)$ satisfies that 
$\gamma_k(x)=O(x^l)$ for any $l\in\N$,
it is easy to see that 
for each $p\in\N$ there exists a positive number $r_p$
such that 
\begin{equation}\label{eqn:12.5}
\gamma(x)\leq \frac{1}{2}|x|^p \quad \text{ for $x\in[-r_p,r_p]$}.
\end{equation}
We take the value of $r_p$ in (\ref{eqn:12.2}) such that
(\ref{eqn:12.5}) holds.

\begin{lemma}
The function $G_1$ satisfies the following properties.
\begin{enumerate}
\item 
There exists a positive number $c$ independent of $p$ 
such that  $G_1(x,y)\geq c$ on $U_1^{(p)}$.
\item 
$\dfrac{\partial^{\alpha+\beta} G_1}
{\partial^{\alpha}x \partial^{\beta} y}(x,y)$ 
can be continuously extended  
to $\overline{U}_1^{(p)}$, 
the closure of the set ${U_1^{(p)}}$,
for all $\alpha,\beta\in\Z_+$.
\end{enumerate}
\end{lemma}

\begin{proof}

(i)\quad 
From (\ref{eqn:12.1}) and (\ref{eqn:12.4}), 
$G_1$ takes the following 
form on $U_1^{(p)}$:
\begin{equation}\label{eqn:12.6}
G_1(x,y)=
u(x,y)\prod_{k=1}^m 
\left(
1-\frac{\gamma_k(x)}{y}
\right).
\end{equation}
From (\ref{eqn:12.5}), we easily see that
\begin{equation*}
\left|\frac{\gamma_k(x)}{y}\right|\leq
\frac{\gamma(x)}{y}\leq 
\frac{1}{2}\frac{x^p}{y}\leq
\frac{1}{2}
\quad \mbox{for $(x,y)\in U_1^{(p)}$},
\end{equation*}
which implies that
\begin{equation*}
\begin{split}
G_1(x,y) \geq \frac{u(0,0)}{2^a}
\quad \mbox{for $(x,y)\in U_1^{(p)}$}.
\end{split}
\end{equation*}

(ii)\quad
Since $G_1$ is a $C^{\infty}$ function
on $U\setminus\{y=0\}$, 
it suffices to show that
every partial derivative of 
$G_1$ can be continuously extended to 
the origin. 
Moreover, from the first equation in (\ref{eqn:12.3}),
it suffices to show that every partial derivative of 
$\psi_j(x,y):=\tilde{\varepsilon}_{m-j}(x)/y^j$ can be continuously
extended to the origin.

Let $(\alpha,\beta)\in\Z_+^2$ be arbitrarily given.
A direct calculation gives 
\begin{equation*}
\frac{\d ^{\a+\b} \psi_j}{\d x^{\a} \d y^{\b}}
=(-1)^{\alpha} \frac{(j+\alpha-1)!}{(j-1)!} 
\frac{\tilde{\varepsilon}_{m-j}^{(\beta)}(x)}{y^{j+\alpha}},
\end{equation*}
where $\tilde{\varepsilon}_{m-j}^{(\beta)}$ 
is the $\b$-th derivative of $\tilde{\varepsilon}_{m-j}$.
Since $\tilde{\varepsilon}_{m-j}^{(\beta)}$ is flat at the origin,
there exists $r_{p,\alpha,\beta}>0$ such that
$
|\tilde{\varepsilon}_{m-j}^{(\beta)}(x)|\leq 
\frac{(j-1)!}{(j+\alpha-1)!}x^{p(j+\alpha+1)}$
for $x\in [-r_{p,\alpha,\beta},r_{p,\alpha,\beta}]$ 
and $j=1,\ldots,m$.
Considering the shape of the set $U_1^{(p)}$, 
we have
$$
\left|
\frac{\d^{\a+\b} \psi_j}{\d x^{\a} \d y^{\b}}\right|
\leq y \quad 
\mbox{ on
$U_1^{(p)}\cap \{|x|\leq r_{p,\alpha,\beta}\}$ 
for $j=1,\ldots,m$.}
$$
Therefore, 
\begin{equation*}
\lim_{U_1^{(p)}\ni (x,y)\to (0,0)}
\frac{\partial^{\a+\b}\psi_j}{\d x^{\a} \d y^{\b}}(x,y)=0
\quad
\mbox{ for $j=1,\ldots,m$}.
\end{equation*}
In particular,  
$\frac{\d ^{\a+\b}\psi_j}{\d x^{\a} \d y^{\b}}$ 
can be continuous up to
the origin for $j=1,\ldots, m$.
\end{proof}


\begin{lemma}
There exists a $C^{\infty}$ function $\tilde{G}_1$ on $U$
such that 
\begin{enumerate}
\item the restriction of $\tilde{G}_1$ to $U_1^{(p)}$ is $G_1$; 
\item $\tilde{G}_1\geq c/2$ on $U$, where $c$ is as in Lemma 12.1.
\end{enumerate}
\end{lemma}

\begin{proof}
It follows from the property (ii) in Lemma~12.2 that 
$G_1$ can be smoothly exetended to the region $U$
by using the Whitney extension theorem 
(\cite{Whi34}, \cite{See64}).
Futhermore, since $G_1\geq c$ on $U_1^{(p)}$, 
the above $C^{\infty}$ extension can be performed so that
$\tilde{G}_1\geq c/2$ on $U$.
\end{proof}


\subsection{Properties of $G_2$}

When a van der Corput-type lemma in Lemma~11.4 is applied
in the next section, 
the following lemma is important.

\begin{lemma}
There exist positive real numbers $R$ and $\mu$ such that 
$$
\frac{\partial^m}{\partial y^m}G_2(x,y)
\geq \mu 
\quad \mbox{ on $[-R,R]^2$.}
$$
\end{lemma}

\begin{proof}
A direct computation gives 
$$
\frac{\partial^m}{\partial y^m}G_2(0,0)=
m! u(0,0) \,\,(>0),
$$
which implies the lemma.
\end{proof}

\section{Proof of Theorem~9.1}

For $p\in\N$, 
let $U_j^{(p)}$ ($j=1,2$) be as in (\ref{eqn:12.2}) and 
let $r_p$ be a positive constant determined by (\ref{eqn:12.5}). 

\subsection{A decomposition of ${\mathcal Z}(s)$}

Let $\chi_p:\R\to[0,1]$ be a cut-off function satisfying that
$\chi_p(x)= 1$ if $|x|\leq r_p/2$ and 
$\chi_p(x)= 0$ if $|x|\geq r_p$.
By using $\chi_p$, the integral ${\mathcal Z}(s)$ in (\ref{eqn:9.7}) 
can be decomposed as
\begin{equation}\label{eqn:13.1}
{\mathcal Z}(s)=
I_1^{(p)}(s)+I_2^{(p)}(s)+J^{(p)}(s)
\end{equation}
with
\begin{equation}\label{eqn:13.2}
\begin{split}
I_j^{(p)}(s)
&=\int_{U_j^{(p)}\cap \R_+^2}
\left|G(x,y)\right|^s x^b
\phi(x,y)\chi_p(x)dxdy \quad 
\mbox{for $j=1,2$},\\
J^{(p)}(s)
&=\int_{\R_+^2}\left|G(x,y)\right|^s x^b
\phi(x,y)(1-\chi_p(x))dxdy.
\end{split}
\end{equation}


\subsection{Meromorphy of the integrals 
$I_1^{(p)}(s)$, $I_2^{(p)}(s)$, $J^{(p)}(s)$}

In order to prove Theorem~9.1,
it suffices to show the following.

\begin{lemma}
Let $p\in\N$.
If the support of $\phi$ is contained in $[-R,R]^2$ 
where $R>0$ is as in Lemma~11.2,
then the following hold:
\begin{enumerate}
\item
The integral $I_1^{(p)}(s)$ 
becomes a meromorphic function
to the whole complex plane. 
Moreover,
its poles exist in the set 
\begin{equation*}
\left\{
-\frac{b+j}{a},-\frac{k}{m},-\frac{b+j+pk}{a+pm}:j,k\in \N
\right\}.
\end{equation*}
\item
The integral $I_2^{(p)}(s)$ becomes a holomorphic function
to the half-plane ${\rm Re}(s)>-(p+b+1)/(mp+a)$.
\item
The integral $J^{(p)}(s)$ becomes a holomorphic function
to the half-plane
${\rm Re}(s)>-1/m$.
\end{enumerate}
\end{lemma}

\begin{proof}

(i)\quad
From (\ref{eqn:12.4}), 
$I_1^{(p)}(s)$ can be expressed as 
\begin{equation}\label{eqn:13.3}
\begin{split}
I_1^{(p)}(s)&=\int
_{U_1^{(p)}\cap \R_+^2}
x^{as+b}y^{ms}\Psi(x,y;s)dxdy,
\end{split}
\end{equation}
where $\Psi:U\times \C\to\C$ is defined by 
\begin{equation*}
\Psi(x,y;s)=
\tilde{G}_1(x,y)^s x^b \phi(x,y)\chi_p(x). 
\end{equation*}
We remark that $\tilde{G}_1=G_1$ on $U_1^{(p)}$. 
\begin{lemma}
The function $\Psi$ satisfies the following properties. 
\begin{enumerate}
\item 
$\Psi(\cdot;s)$ is a $C^{\infty}$ function on $U$
for all $s\in\C$.
\item 
$\dfrac{\d^{\alpha+\beta}\Psi}
{\d x^{\alpha}\d y^{\beta}}(x,y;\cdot)$ 
is an entire function 
for all $(x,y)\in U_1^{(p)}$ and $(\alpha, \beta)\in \Z_+^2$.
\end{enumerate}
\end{lemma}

\begin{proof}
Since $x^b \phi(x,y)\chi_p(x)$ 
does not give any essential influence, 
it suffices to show that the function 
$\tilde{G}_1(x,y)^s$ satisfies the properties (i), (ii) in the lemma. 


Every partial derivative of $\tilde{G}_1(x,y)^s$
with respect to $x,y$ can be expressed as the sum of
$s(s-1)\cdots (s-k+1) \tilde{G}_1(x,y)^{s-k}$ for $k\in\N$
multiplied by polynomials of the partial derivatives
of $\tilde{G}_1(x,y)$ with respect to $x,y$. 
Applying Lemma~12.2 to this expression, 
we can easily see that 
$\tilde{G}_1(x,y)^s$ satisfies the properties in the lemma. 
\end{proof}

Since 
$\Psi$ satisfies the same properties as those of 
$\Phi$ in Lemma 11.2, 
the integral $I_1^{(p)}(s)$
can be analytically continued 
as a meromorphic function to the whole complex plane
and, moreover, its poles are contained in the set
\begin{equation*}
\left\{
-\frac{b+j}{a},-\frac{k}{m},-\frac{b+j+pk}{a+pm}:j,k\in \N
\right\}.
\end{equation*}


\vspace{.5 em}

(ii) \quad
From (\ref{eqn:12.4}), 
$I_2^{(p)}(s)$ can be expressed as
\begin{equation*}
I_2^{(p)}(s)
= 
\int_{U_2^{(p)}\cap\R_+^2}
x^{as+b}
\left|G_2(x,y)
\right|^{s}
\varphi(x,y)\chi_p(x)dxdy.
\end{equation*}
It is easy to see that
\begin{equation}\label{eqn:13.4}
|I_2^{(p)}(s)|
\leq C_p 
\int_{0}^{r_p}x^{a{\rm Re}(s)+b}
\left(
\int_{0}^{x^{p}}
\left|G_2(x,y)
\right|^{{\rm Re}(s)}dy
\right)
dx,
\end{equation}
where $C_p:=
\sup_{(x,y)\in U_2^{(p)}}(|\varphi(x,y) \chi_p(y)|)$.
Since  
Lemma 11.4 can be applied to the integral 
with respect to the variable $x$ 
in (\ref{eqn:13.4}) from Lemma~12.3, 
if ${\rm Re}(s)>-1/m$,
then
\begin{equation}\label{eqn:13.5}
\begin{split}
|I_2^{(p)}(s)|
&< C_p {\mathcal C}({\rm Re}(s),a)\mu^{{\rm Re}(s)}
\int_{0}^{r_p}
x^{a{\rm Re}(s)+b}
\left(x^{p}\right)^{1+m{\rm Re}(s)}
dx\\
&= C_p {\mathcal C}({\rm Re}(s),a)\mu^{{\rm Re}(s)}
\int_{0}^{r_p}
x^{(mp+a){\rm Re}(s)+p+b}dy,
\end{split}
\end{equation}
where ${\mathcal C}(\cdot,\cdot)$
is as in Lemma~11.4 and $\mu$ is as in Lemma~12.3.  
The second integral in (\ref{eqn:13.5}) converges
on the half-plane ${\rm Re}(s)>-(p+b+1)/(mp+a)$, 
on which
$I_2^{(p)}(s)$ becomes a holomorphic function. 

\vspace{.5 em}

(iii)\quad
In a similar fashion to the case of 
integral $I_2^{(p)}(s)$, 
we have
\begin{equation*}
\begin{split}
|J^{(p)}(s)|\leq \tilde{C}_p
\int_{r_p/2}^{R}
x^{a{\rm Re}(s)+b}
\left(
\int_{0}^{R} |G_2(x,y)|^{{\rm Re}(s)}dy
\right)
dx,
\end{split}
\end{equation*}
where 
$\tilde{C}_p:=\sup_{(x,y)\in[0,R]\times [r_p/2,R]}
(|\varphi(x,y) (1-\chi_p(x))|).$
Applying Lemma~11.4, 
we can show that
if 
${\rm Re}(s)>-1/m$,
then
\begin{equation*}
\begin{split}
|J^{(p)}(s)|
\leq  \tilde{C}_p {\mathcal C}({\rm Re}(s),a)\mu^{{\rm Re}(s)}
R^{1+m{\rm Re}(s)}
\int_{r_p/2}^{R}x^{a{\rm Re}(s)+b}dx.
\end{split}
\end{equation*}
Since the above integral converges for any $s\in \C$, 
$J^{(p)}(s)$
can be analytically continued as a holomorphic function
to the half-plane ${\rm Re}(s)>-1/m$.
\end{proof}

\subsection{Proof of Theorem 9.1}

From (\ref{eqn:13.1}), (\ref{eqn:13.2}), 
Lemma 13.1 gives Theorem~9.1 by 
letting $p$ tend to infinity.


\section{Some comments}


\subsection{Revisited Section~2 by means of formal power series}

Since many concepts related to $f\in \R C^{\infty}((x,y))$ 
in Section~2 can be determined
by its formal Taylor series only, 
we will rewrite these concepts and the related results by
using the formal series ${\bf f}\in\R[[x,y]]$,
where ${\bf f}$ is the Taylor series of $f$. 

For $f\in \R C^{\infty}((x,y))$, 
let $\rho(f)$ be the Taylor series of $f$. 
From this, 
the map $\rho: \R C^{\infty}((x,y)) \to \R[[x,y]]$ is 
defined, which is called the {\it Borel map}. 
It follows from E. Borel's theorem that 
the map $\rho$ is surjective. 
On the other hand, 
the restriction of $\rho$ to $\R\{x,y\}$
is injective, but $\rho$ itself is not so. 
This subtle difference is important in our investigation. 
 
The Newton polygon $\check{\Gamma}_+({\bf f})$ 
of ${\bf f}\in\R[[x,y]]$ 
can be well-defined by $\check{\Gamma}_+({\bf f})
=\Gamma_+(f)$ where $\rho(f)={\bf f}$. 
Furthermore, 
the maps $\check{\delta}_0, \check{\mu}_0:
\R[[x,y]]\to\R_+$  can be also well-defined by 
\begin{equation}
\check{\delta}_0({\bf f}):=\delta_0(f),\quad
\check{\mu}_0({\bf f}):=\mu_0(f) \quad\quad 
\mbox{ for } {\bf f}\in \R[[x,y]],
\end{equation}
where $\rho(f)={\bf f}$.

Next, let us consider the quantities ${\mathfrak h}_0(f)$ and
${\mathfrak m}_0(f)$.
Theorem~2.2, given by Greenblatt \cite{Gre06}, 
shows that ${\mathfrak h}_0(f)$
is determined by the formal Taylor seies of $f$ only, which implies 
that 
the map 
$\check{{\mathfrak h}}_0: \R[[x,y]]\to\R_+$ 
can be well-defined by
\begin{equation}
\check{{\mathfrak h}}_0({\bf f}):={\mathfrak h}_0(f) 
\quad\quad 
\mbox{ for } {\bf f}\in \R[[x,y]],
\end{equation}
where $\rho(f)={\bf f}$. 
Theorem~2.2 can be rewritten as 
\begin{theorem}
$\check{{\mathfrak h}}_0({\bf f})
=1/\check{\delta}_0({\bf f})$ holds
for every ${\bf f}\in \R[[x,y]]$ with ${\bf f}\neq \bar{0}$.
\end{theorem} 
On the other hand, the example (\ref{eqn:2.11}) implies that 
${\mathfrak m}_0(f)$ is not always determined by 
the formal Taylor series of $f$ only (see Remark~2.3). 
Now, we define the map   
$\check{{\mathfrak m}}_0:\R[[x,y]]\to\R_+$ by 
\begin{equation}
\check{{\mathfrak m}}_0({\bf f}):=
\inf\{{{\mathfrak m}}_0(f):\rho(f)={\bf f}\}
\quad\quad 
\mbox{ for } {\bf f}\in \R[[x,y]].
\end{equation}
From Theorem~2.5 (ii), 
the following problem is naturally raised.
\begin{problem}
Does
$\check{{\mathfrak m}}_0({\bf f})=
1/\check{\mu}_0({\bf f})$ 
hold, if $\check{\mu}_0({\bf f})\geq 2$? 
\end{problem}
If the above problem is affirmatively solved,  
then a strong optimality of the inequality in Theorem~2.5 (ii) 
is shown. 
It is shown in \cite{KaN19}, \cite{Nos21} that  
if the Newton polygon of ${\bf f}$ takes the form
$\{(\alpha,\beta):
\alpha\geq a, \,\, \beta\geq b\}$
where $a,b\in\N$ satisfy $2\leq a < b$, 
then the above equality holds.  


\subsection{Non-polar singularity of local zeta functions}

From Theorem~2.5 and the example (\ref{eqn:2.11}), 
it is possible that
local zeta functions have a singularity 
different from the pole if $\mu_0(f)\geq 2$.
It is an interesting issue to investigate
what kinds of singularities these local zeta functions have.
At present, 
there seems to be  no investigation of this issue
except one in specific cases in \cite{KaN19}, \cite{Nos21}.
Let us roughly explain the situation which has been known 
in the above papers. 

In \cite{KaN19}, 
the behavior at $s=-1/\mu_0(f)(=-1/b)$ 
of the restriction of the local zeta function 
associated with (\ref{eqn:2.11}) to the real axis
is exactly computed. 
(We assume that the parameters $a,b,p,q$ in (\ref{eqn:2.11}) 
satisfy the conditions in Section~2.3.)
Note that this example 
satisfies the condition $\mu_0(f)=\delta_0(f)(=b)$. 
Though it depends on the parameters, 
the above behavior is different from that at the poles
in any cases, 
which implies that the local zeta function associated with 
(\ref{eqn:2.11}) has a non-polar singularity at the point $s=-1/b$. 
Unfortunately,  the above information cannot  
determine whether this singularity is isolated or not. 
If this singularity was isolated, then it must be an essential singularity.
On the other hand, 
we cannot deny that the line ${\rm Re}(s)=-1/\mu_0(f)$ 
is the natural boundary for its local zeta function.

More recently, 
T. Nose \cite{Nos21}  
shows the existence of a $C^{\infty}$ function $f$ with 
$\mu_0(f)<\delta_0(f)$ 
such that its local zeta function 
has a non-polar singularity at the point $s=-1/\mu_0(f)$.

It is expected to understand more detailed properties of non-polar singularities 
of local zeta functions in the future.

\bigskip

{\bf Acknowledgements.} 
The author is grateful to Toshihiro Nose
for useful discussion for a substantial period  time.
This work was supported by 
JSPS KAKENHI Grant Numbers JP20K03656, JP20H00116.


\end{document}